\documentclass[sn-mathphys,Numbered]{sn-jnl}% Math and Physical Sciences Reference Style
\usepackage{graphicx}%
\usepackage{multirow}%
\usepackage{amsmath}
\usepackage{amssymb,amsfonts,subfigure}%
\usepackage{mathrsfs,bm,float,geometry,array,booktabs}
\usepackage{multirow}
\usepackage{amsthm}%
\usepackage{mathrsfs}%
\usepackage[title]{appendix}%
\usepackage{xcolor}%
\usepackage{textcomp}%
\usepackage{manyfoot}%
\usepackage{booktabs}%
\usepackage{algorithm}%
\usepackage{algorithmicx}%
\usepackage{algpseudocode}%
\usepackage{listings}%
\theoremstyle{thmstyleone}%
\theoremstyle{thmstyletwo}%

\theoremstyle{thmstylethree}%
\usepackage[justification=centering]{caption}

\newcommand{\by}{\bm{y}}

\newcommand{\bx}{\bm{x}}
\newcommand{\bp}{\bm{p}}

\newcommand{\bn}{\bm{n}}

\newcommand{\bk}{\bm{k}}
\newcommand{\bj}{\bm{j}}

\newcommand{\bP}{\bm{P}}

\newcommand{\bK}{\bm{K}}

\newcommand{\bmm}{\bm{m}}
\newcommand{\bell}{\bm{\ell}}

\newcommand{\blam}{\bm{\lambda}}

\newcommand{\bbR}{\mathbb{R}}
\newcommand{\bbC}{\mathbb{C}}
\newcommand{\bbZ}{\mathbb{Z}}
\newcommand{\bbQ}{\mathbb{Q}}

\newcommand{\bbN}{\mathbb{N}}

\newcommand{\bbT}{\mathbb{T}}

\newcommand{\bbP}{\mathbb{P}}

\newcommand{\calC}{\mathcal{C}}

\newcommand{\calS}{\mathcal{S}}

\newcommand{\calV}{\mathcal{V}}

\newcommand{\calP}{\mathcal{P}}

\newcommand{\calI}{\mathcal{I}}

\newcommand{\calT}{\mathcal{T}}
\newcommand{\calX}{\mathcal{X}}

\newcommand{\hv}{\hat{v}}

\newcommand{\hU}{\hat{U}}
\newcommand{\hu}{\hat{u}}
\newcommand{\hf}{\hat{f}}

\newcommand{\tU}{\tilde{U}}

\newcommand{\tu}{\tilde{u}}

\newcommand{\QP}{\mbox{QP}}
\newcommand{\Err}{\mbox{Err}}
\newcommand{\Ord}{\mbox{Ord}}

\newtheorem{thm}{Theorem}[section]
\newtheorem{lemma}[thm]{Lemma}

\newtheorem{defy}[thm]{Definition}

\newcommand\tbbint{{-\mkern -16mu\int}}

\newcommand\dbbint{{-\mkern -19mu\int}}

\newcommand\bbint{
	{\mathchoice{\dbbint}{\tbbint}{\tbbint}{\tbbint}}
}

\allowdisplaybreaks[4]

\begin{document}

\title[Article Title]{High-accuracy numerical methods and convergence analysis for Schr\"{o}dinger equation with incommensurate potentials}


\author[1]{\fnm{Kai} \sur{Jiang}}\email{kaijiang@xtu.edu.cn}

\author[1]{\fnm{Shifeng} \sur{Li}}\email{shifengli@smail.xtu.edu.cn}

\author[1]{\fnm{Juan} \sur{Zhang}}\email{zhangjuan@xtu.edu.cn}

\affil[1]{Department of Mathematics and Computational
	Science, Xiangtan University, Xiangtan, Hunan,
	411105, P. R. China}

%\affil[2]{\orgdiv{Department}, \orgname{Organization}, \orgaddress{\street{Street}, \city{City}, \postcode{10587}, \state{State}, \country{Country}}}
%
%\affil[3]{\orgdiv{Department}, \orgname{Organization}, \orgaddress{\street{Street}, \city{City}, \postcode{610101}, \state{State}, \country{Country}}}


\abstract{Numerical solving the Schr\"{o}dinger equation with incommensurate potentials presents a great challenge since its solutions could be space-filling quasiperiodic structures without translational symmetry nor decay. In this paper, we propose two high-accuracy numerical methods to solve the time-dependent quasiperiodic Schr\"{o}dinger equation. Concretely, we discretize the spatial variables by the quasiperiodic spectral method and the projection method, and the time variable by the second-order operator splitting method. The corresponding convergence analysis is also presented and shows that the proposed methods both have exponential convergence rate in space and second order accuracy in time, respectively. Meanwhile, we analyse the computational complexity of these numerical algorithms. One- and two-dimensional numerical results verify these convergence conclusions, and demonstrate that the projection method is more efficient.
}




\maketitle

\section{Introduction}
\label{sec:Intro}

%\subsection{Background}

%An important task in molecular dynamics and other related  fields is to solve the time-dependent Schr\"{o}dinger equation (TSE). 
In this paper, we consider the time-dependent quasiperiodic Schr\"{o}dinger equation (TQSE)
\begin{align}
i\frac{\partial u(\bx,t)}{\partial t}
=-\Delta u(\bx,t)+V(\bx)u(\bx,t)+f(\bx,t),~~(\bx,t)\in\bbR^d\times[0,T],
\label{eq:QSE}
\end{align}
where $\Delta=\partial^{2}_{x_1}+\cdots +\partial^{2}_{x_d}$ is the Laplacian operator, $V(\bx)$ is the incommensurate potential, $f(\bx,t)$ is an external field function and $u(\bx,t)$ is quasiperiodic in $\bx$. The initial value $u(\bx, 0)=u_0(\bx)$.
%From mathematical point of view, QSE is a partial differential equation, containing one (or more) spatial variables and a temporal variable. 

%{\bf The background of Schr\"{o}dinger equation :} The Schr\"{o}dinger equation represents the  interaction  of a system of quantum mechanical particles with classical monochromatic radiation.  

%{\bf The background of quasiperiodic Schr\"{o}dinger equation:}
%In comparison to the periodic case, which is well understood and described in
%the Floquet theory (see for example [?]), the quasi-periodic case has shown to be
%elicate to analyze and the existing results reveal a very rich flora of various different phenomena. The main difficulty in the analysis is the presence of resonances, which appear in the form of small divisors.

%The research on Schr\"{o}dinger equation in periodic system has been relatively mature both in application and in mathematical theory, but it is still a challenge for this kind of equation in quasiperiodic system. The research of quasiperiodic Schr\"{o}dinger system has gained a lot of new understanding and scientific value both in physical science and in materials.
% that holds the same central importance as Newton's laws of motion in classical mechanics.
%In these systems, quasiperiodic phenomena may occur both in time and space.

The Schr\"{o}dinger equation is a fundamental equation in quantum mechanics. 
The study of the Schr\"{o}dinger equation with periodic potentials has reached a relatively mature stage \cite{Dirac1958principles,Jahnke2000error}. 
In the last few decades, the (nonlinear) quasiperiodic Schr\"{o}dinger equations (QSEs) and quasiperiodic Schr\"{o}dinger operators, have been attracted much attention due to the fascinating phenomena such as Anderson localization, mobility edge, topological phase transition \cite{Jitomirskaya1999Metal-insulator, Nixon2010Electronic,Zhou2014Particle-hole, wang2020localization, Kohmoto1983Metal-Insulator, Lahini2009observation, Cao2018unconventional}. 
Many mathematical works have been presented to study these incommensurate systems. 
Kuksin and P\"{o}schel's earlier work established the existence of quasiperiodic solutions in time direction for one-dimensional nonlinear QSEs by extending the Kolmogorov-Arnold-Moser (KAM) method to infinite dimensions \cite{Kuksin1996Invariant}.
%{\color{red}Therefore, the small divisors problem in the quasiperiodic Schr\"{o}dinger system can be handled via KAM theory \cite{Liang2005Quasi-periodic}.} 
Bourgain developed the modified Craig-Wayne method to prove the existence of quasiperiodic solutions for nonlinear QSE with periodic or spatial Dirichlet boundary conditions \cite{Bourgain1994Construction,Bourgain1998Quasi-periodic}.
A recent work by Berti and Bolle investigated the existence of time-quasiperiodic solutions of $\bbT^d$-QSE  in $C^{\infty}$-space by the Nash-Moser iteration \cite{Berti2012Sobolev, Berti2013Quasi-periodic}. Besides, Cong \textit{et al.} studied Anderson localization of the nonlinear time-quasiperiodic Schr\"{o}dinger equation by using the Birkhoff normal form transform \cite{Cong2023Longtime}. While the time-quasiperiodic Schr\"{o}dinger equation has received considerable attention, there are few theoretical results available for the spatial quasiperiodic Schr\"{o}dinger equation. Recently, Wang has constructed spatial quasiperiodic solutions to the nonlinear Schr\"{o}dinger equation (NLS) on $\bbR^d$ \cite{Wang2020Space}. Furthermore, the time-space quasiperiodic solutions for the non-integrable NLS on $\bbR$ have been analyzed in \cite{Wang2022Infinite}. 
For the spectral theory of quasiperiodic Schr\"{o}dinger operators, one can refer to \cite{Avila2009spectrum, Marx2017Dynamics}.

%Meanwhile, in the field of spectral theory for quasiperiodic Schr\"{o}dinger operators, Avila has proved that the spectrum of such operators is a Cantor set \cite{Avila2009spectrum}. Moreover, he collaborated with Jitomirskaya to solve Ten Martini problem \cite{Avila2009ten}.
%Recently, there has been a remarkable development of arithmetically spectral transition (singular spectrum and Anderson localization) for quasiperiodic Schr\"{o}dinger operator \cite{Morozov2014Complete, Marx2017Dynamics}.

However, numerically solving the QSE is still a great challenge since their solutions could be quasiperiodic, globally ordered structure without translational symmetry nor decay. 
The periodic approximation method (PAM) has been used to obtain an approximate periodic solution within a finite domain \cite{wang2020localization}.
Besides, the PAM is also used to solve and analyze the eigenvalues of Schr\"{o}dinger operators \cite{Jitomirskaya2012Analytic, Damanik2014isospectral}.
However, the PAM is often unsatisfactory in terms of algorithm accuracy and convergence rate, see \cite{Jiang2022approximation} for details. %The linear convergence rate of PAM is usually obtained with great computational cost.
Therefore, there is still a lack of highly precise and efficient numerical algorithms for solving TQSE to obtain the quasiperiodic solution, especially for the case of arbitrary dimensions. Recently, two effective numerical methods, quasiperiodic spectral method (QSM) and projection method (PM), have been proposed to solve quasiperiodic systems \cite{jiang2014numerical, jiang2018numerical}. 
The corresponding function approximation rate analysis of QSM and PM has been given in \cite{Jiang2022Numerical}, respectively.
The PM has been also applied to quasiperiodic Schr\"{o}dinger eigenvalue systems\,\cite{Gao2023Pythagoras}.
Extensive studies have demonstrated that the PM can achieve high accuracy in computing various quasiperiodic systems \cite{jiang2015stability, cao2021computing, li2021numerical, jiang2022tilt,wang2022Effective}.
However, the corresponding numerical analysis is lacking.
These theoretical results and applications illuminate our problems. 
The purpose of this paper is to study how to efficiently solve the spatially quasiperiodic solutions of higher-dimensional quasiperiodic Schr\"{o}dinger equation to high accuracy, and to establish the corresponding convergence analysis.
Concretely, we apply QSM and PM to solve the quasiperiodic solution of TQSE \eqref{eq:QSE} and analyze their computational complexity. Meanwhile, a rigorous error analysis shows that both algorithms have exponential convergence rates. One- and two-dimensional numerical examples are given to verify the effectiveness of the proposed algorithms. Furthermore, we can obtain quasiperiodic solutions and show that the PM is an efficient and high-precision algorithm for solving the TQSE.

%\subsection{Notation and Outline}
%{\bf Auxiliary notations:} 
% {\color{blue}
	%	As standard, for any integer number $1\leq p\leq \infty $, we denote by $L^p( \Omega)$ the Lebesgue space of complex-valued functions on $\Omega$
	%	and $W^{\alpha,p}(\Omega)$ all functions with partial derivatives up to order $m\leq 1$ contained in $L^p( \Omega)$.}

The rest of the paper is organized as follows. Section \ref{sec:Numericalmethod} proposes two methods, the PM and the QSM, coupled with the second-order operator splitting scheme, to solve TQSE \eqref{eq:QSE}. We also give their numerical implementation and computational complexity analysis. Section \ref{sec:converanaly} introduces quasiperiodic Hilbert spaces and gives the convergence analysis of the proposed methods. Section \ref{sec:Num} presents numerical results to further validate the theoretical analysis. Finally, the conclusion of this paper is given in Section \ref{sec:diss}.

\section{Numerical methods} 
\label{sec:Numericalmethod}

Throughout, we make use of the following notations. Let $\Omega_L=[-L,L]^d \subset \bbR^d$ and $\vert\Omega_L\vert=(2L)^d$. 
%$\bbN_0=\bbN \setminus \{0\}$ represents the set of nonnegative integer numbers. 
For any vector $\bx\in\bbR^d$, we define 
$\Vert \bx\Vert^2=\sum_{j=1}^{d}\vert x_j\vert^2 $ and 
$\vert \bx\vert=\sum_{j=1}^{d}\vert x_j\vert $.
For any multi-index $\mu=(\mu_1,\cdots,\mu_d)\in\bbN^d$ and $\bx\in \bbR^d$, let
$\partial^\mu_{\bx}=\partial^{\mu_1}_{x_1}\cdots \partial^{\mu_d}_{x_d}$. 
We present the definition of the quasiperiodic function.
\begin{defy}
A matrix $\bP\in\bbR^{d\times n}$ is the projection matrix, if it belongs to the set  
$\bbP^{d\times n} :=\{\bP=(\bp_1,\cdots,\bp_n)\in\bbR^{d\times n}: \bp_1,\cdots,\bp_n  ~\mbox{are~} \bbQ\mbox{-linearly independent}\}.$
	%and $\bP=(\bp_1,~\bp_2,\cdots,\bp_n)\in \bbP^{d\times n}$ be the projection matrix.
\end{defy}
\begin{defy}
	\label{def:quasiperiodic}
	A $d$-dimensional function $u(\bx)$ is quasiperiodic if there exists an $n$-dimensional periodic function $u_p$ and a projection matrix $\bP\in \bbP^{d\times n}$, such that $u(\bx)=u_p(\bP^T \bx)$ for all $\bx\in\bbR^d$.
\end{defy}

In particular, when $n=d$ and $\bp_1,~\bp_2,\cdots,\bp_d$ form a basis of $\bbR^d$,
$u(\bx)$ is periodic. For convenience, we refer to $u_p$ in Definition \ref{def:quasiperiodic} as the parent function of $u$. 
The continuous Fourier-Bohr transformation of $u(\bx)$ is
\begin{align}
	\hu_{\blam}= \lim_{L\rightarrow\infty} \frac{1}{\vert \Omega_L\vert} 
	\int_{\Omega_L} u(\bx)e^{-i\blam\cdot \bx}\,d\bx:=\bbint u(\bx)e^{-i\blam\cdot \bx}\,d\bx,
	\label{eq:transform-FC}
\end{align}
where $\blam\in\bbR^d$.
The Fourier series associated to $u(\bx)$ is
\begin{align}
	u(\bx)\sim\sum_{j=1}^{\infty} \hu_{\blam_j} e^{i\blam_j\cdot \bx},\label{eq:Fouriercoeff}
\end{align}
where $\blam_j\in \sigma(u)=\{\blam: \blam = \bP\bk,~\bk\in \bbZ^n \}$ are Fourier exponents, $\hu_{\blam_j}$ computed by \eqref{eq:transform-FC} are Fourier coefficients. 
If the Fourier series \eqref{eq:Fouriercoeff} is absolutely
convergent, it is also uniformly convergent. Therefore, when $\sum_{j=1}^{\infty} \vert \hu_{\blam_j}\vert <+\infty$, we have 
\begin{align*}
	u(\bx)=\sum_{j=1}^{\infty} \hu_{\blam_j} e^{i\blam_j\cdot \bx}.
\end{align*}
Moreover, the Parseval equality
\begin{align}
	\sum_{j=1}^{\infty}\vert \hu_{\blam_j}\vert^2
	=\bbint \vert u(\bx)\vert^2\,d\bx
	\label{eq:parseval}
\end{align}
is true.

% For a given quasiperiodic potential function
% \begin{align}
	% 	V(\bx)=\sum_{\blam\in \sigma(V)}\hat{V}_{\blam}e^{i\blam\cdot \bx},
	% 	\label{eq:defV}
	% \end{align}
% assume that the wave function $u$ in TQSE \eqref{eq:QSE} has the following Fourier series expansion
% \begin{align*}
	% 	u(\bx,t)=\sum_{\blam\in\sigma(u)} \hu_{\blam}(t)e^{i\blam\cdot \bx}.
	% \end{align*}

%The proposed methods are using PM or QSM for the space and second order Strang splitting scheme for the time, detailed information on PM (or QSM) and splitting methods is found, for instance, in \cite{Jiang2022Numerical,Jahnke2000error}.
For simplicity, we consider the homogeneous TQSE, \textit{i.e.,} $f=0$, in the following analysis.
For inhomogeneous $f\neq 0$, the presented analysis below can be easily extended.
Denote the operator $A=-\Delta$ and for a given quasiperiodic potential 
\begin{align}
	V(\bx)=\sum_{\blam\in \sigma(V)}\hat{V}_{\blam}e^{i\blam\cdot \bx},
	\label{eq:defV}
\end{align}
TQSE \eqref{eq:QSE} becomes
\begin{align}
	\frac{\partial u}{\partial t}=-iA u-iVu.
	\label{eq:QSE-re}
\end{align}
The formal solution of \eqref{eq:QSE-re} is
\begin{align*}
	u(\bx,t)=e^{-it(A+V)}u_0=\calT u_0,~~(\bx,t)\in \bbR^d\times [0,T].
\end{align*}
Next, we employ QSM and PM to discretize TQSE \eqref{eq:QSE-re} in space direction, and the second-order operator splitting (OS2) method in time direction.

%\subsection{QSM and PM}
\subsection{Spatial discretization}

%The rigorous mathematical framework and theoretical error analysis of QSM and PM are gave in  \cite {Jiang2022Numerical}, respectively. 

%Moreover, many applications have shown that the PM is highly accurate in computing quasiperiodic systems\,\cite{jiang2015stability,
	%	zhou2019plane, cao2021computing, li2021numerical, jiang2022tilt}. 

%In \cite{Jiang2022Numerical}, the connection between quasiperiodic function and higher dimensional periodic function is established using Birkhoff's ergodic theorem.
We first introduce the QSM and the PM, respectively.
For a positive integer $N\in\bbN_0=\bbN \setminus \{0\}$ and the given projection matrix $\bP\in\bbP^{d\times n}$, denote
\begin{align*}
	\bK_N^n=\{\bk=(k_j)_{j=1}^n \in\bbZ^n: \, -N \leq  k_j < N \},
\end{align*}
and $\sigma_{N}=\{\blam = \bP\bk: \bk\in \bK_N^n \}.$
The order of set $\sigma_{N}$ is $\# (\sigma_{N})=(2N)^n:=D$.

\subsubsection{QSM method}

The QSM approximates the quasiperiodic function $u(\bx)$ by the truncation operator $\calP_N$, \textit{i.e.,} 
\begin{align*}
	u(\bx)\approx	
	\calP_N u(\bx) =\sum_{\blam_j \in \sigma_N }
	\hu_{\blam_j} e^{i \blam_j\cdot \bx},~~\bx \in \bbR^d,
\end{align*}
where $\hu_{\blam_j}$ is obtained by the continuous Fourier-Bohr transformation of $u(\bx)$.
The function approximation theory of QSM can be found in \cite{Jiang2022Numerical}.

%For QSE \eqref{eq:QSE}, we consider using PM in the spatial direction.
%Let's first review the discrete space and interpolation functions of PM.

%For periodic systems, we can use the spectral collocation method to obtain Fourier coefficients.

\subsubsection{PM method}
An alternative way is PM, seizing the fact that the quasiperiodic system is defined on the irrational manifold of higher-dimensional torus. Concretely, the PM efficiently computes the Fourier coefficient of the periodic parent function on a higher-dimensional torus in a pseudo-spectral manner. Then, the quasiperiodic structure can be obtained by projecting the high-dimensional periodic structure onto a corresponding irrational manifold.
Let $\bbT^n=(\bbR/2\pi \bbZ)^n$ be the $n$-dimensional torus. To discretize $\bbT^n$, we consider a fundamental domain $[0,2\pi)^n$ and assume that the discrete node along each dimension is the same, \textit{i.e.,} $N\in\bbN_0$. Then $[0,2\pi)^n$ is discretized by 
grid points $\by_{\bj} =(y_{1,j_1},
y_{2,j_2},\dots, y_{n,j_n})$, where $y_{1,j_1}=j_1 h$, $y_{2,j_2}=j_2 h, \dots,
y_{n,j_n}=j_n h$, $0\leq j_1,j_2,\dots, j_n < 2N$, with the spatial discretization size $h=\pi/N$. 
The discrete $n$-dimensional torus $\bbT^n_N$ can be obtained by periodic extending these grid points $\by_{\bj}$. The grid periodic function space defined on $\bbT^n_N$ is
\begin{align*}
	\mathcal{G}_N := \{U: \bbZ^n \mapsto \bbC: ~ U~ \mbox{is}~ \bbT^n_N \mbox{-periodic} \}.
\end{align*}
Given any periodic grid functions $U_1, U_2\in\mathcal{G}_N$, the $\ell^2$-inner product is
\begin{align}
	\langle U_1, U_2 \rangle_N = \frac{1}{(4\pi N)^n}\sum_{\by_{\bj}\in\bbT^n_N}
	U_1(\by_{\bj})\overline{U}_2(\by_{\bj}).
\end{align}
For $\bk, \bell\in\bbZ^n$, we have the discretize orthogonality
\begin{align}
	\langle e^{i\bk\cdot \by_{\bj} }, e^{i\bell\cdot \by_{\bj} } \rangle_N =
	\begin{cases}
		1,~~\bk=\bell + 2N\bmm,~\bmm \in \bbZ^n,\\
		0,~~\mbox{otherwise}.
	\end{cases}
	\label{eq:disOrth}
\end{align}
%\begin{align*}
%	K_N^n=\big\{ \by=(y_{j})_{j=1}^n\in\bbT^n:~ y_{j}\in j\pi h/N,~ %h=0,1,\ldots,2N-1  \big\}
%\end{align*}
%referred to as nodes and $\mbox{dim}K_N^n=(2N)^n$.
The discrete Fourier coefficient of $U\in \mathcal{G}_N$ is
\begin{align}
	\tU_{\bk} = \langle U, e^{i \bk\cdot \by_{\bj}} \rangle_N,~~~
	\bk \in \bK_N^n.
	\label{eq:disFourierCoeff}
\end{align}
The PM designates $\tU_{\bk}$ as the Fourier-Bohr coefficient $\tu_{\blam}$, $\blam =\bP\bk$. Then we obtain the discrete Fourier-Bohr expansion of $u(\bx)$ is
\begin{align*}
	u(\bx_{\bj})=\sum_{\blam\in \sigma_N} \tilde{u}_{\blam}e^{i\blam\cdot \bx_{\bj}},
\end{align*}
where collocation points $\bx_{\bj}\in \bbQ_N=\{\bx_{\bj}=\bP\by_{\bj}$, $\by_{\bj}\in \bbT^n_N\}$.
The trigonometric interpolation of $u$ is
\begin{align}
	I_N u(\bx)=\sum_{\blam\in \sigma_N} \tu_{\blam}e^{i\blam\cdot\bx}.
	\label{eq:tri-quasi}
\end{align}
Consequently, let $I_N u(\bx_j) \approx u(\bx_j)$.
Recent function approximation theory has shown that the PM has exponential convergence \cite{Jiang2022Numerical}. 
From the implementation, it is apparent that the PM can use the $n$-dimensional fast Fourier transform (FFT) to obtain Fourier coefficients but the QSM could not. 

\subsection{Second-order operator splitting (OS2) method}

The OS2 scheme is one of the most popular numerical methods of solving Schr\"{o}dinger equations \cite{Strang1968construction}. The basic idea of OS2 method is splitting
TQSE \eqref{eq:QSE-re} into two subproblems
\begin{align}
	\frac{\partial u}{\partial t}=-iA u,~~
	\frac{\partial u}{\partial t}=-iVu.
	\label{eq:os2subproblems}
\end{align}
For the time grid points $0 = t_0 < t_1 <\cdots < t_M = T$, where $t_m=m\tau$, $m=0,1,\dots, M$, and the time step size $\tau=T/M$, the OS2 method numerically approximates the solution by the recurrence relation
\begin{align}
	u^{m}=e^{-\frac{i}{2}\tau A} e^{-i\tau V} e^{-\frac{i}{2}\tau A}u^{m-1}
	=\calS^m u^0\approx u(\cdot, t_m),~~
	1\leq m\leq M,
	\label{eq:splitting}
\end{align}
where $\calS=e^{-\frac{i}{2}\tau A} e^{-i\tau V} e^{-\frac{i}{2}\tau A}$ and $u^0=u(\cdot,0)$.

%{\color{red}The numerical realization steps are as follows:}

\subsection{Numerical implementation}

In numerical implementation, the semi-discrete subproblems of \eqref{eq:os2subproblems} can be solved by QSM and PM, in terms of QSM-OS2 and PM-OS2, respectively. The detailed implementation to solve TQSE \eqref{eq:QSE-re} and computational complexity analysis are shown below.

\subsubsection{QSM-OS2 method}

From $t_m$ to $t_{m+1}$, the QSM-OS2 involves three steps.

$\bullet$ \textbf{QSM-OS2-Step 1.} For $t\in [t_m, t_m+\tau/2]$, consider the ordinary differential equation
\begin{align}
	\frac{\partial u}{\partial t}=-iA u,
	\label{eq:subode1}
\end{align}
with initial value $u(\cdot, t_m)$. Therefore, we have
\begin{align*}
u(\bx, t_m+\frac{\tau}{2}) &\approx	\phi(\bx, t_m)
	=e^{-\frac{i}{2}\tau \check{A}_N} u(\bx, t_m)
	\\
	&=e^{-\frac{i}{2}\tau A}\calP_N u(\bx, t_m)
	=\sum_{\blam\in\sigma_N(u)} \hu_{\blam}(t_m)e^{-\frac{i}{2}\tau \Vert \blam\Vert^2}e^{i\blam\cdot \bx}.
\end{align*}
Denote $\hat{\phi}_{\blam}(t_m)= \hu_{\blam}(t_m)e^{-\frac{i}{2}\tau \Vert \blam\Vert^2}$.

$\bullet$ \textbf{QSM-OS2-Step 2.} For $t\in [t_m, t_{m+1}]$, consider the ordinary differential equation
\begin{align*}
	\frac{\partial u}{\partial t}=-iV u,
\end{align*}
with initial value $\phi(\bx, t_m)$. From the definition of $V$ in \eqref{eq:defV}, we have
\begin{align*}
	\psi(\bx, t_m)=\phi(\bx, t_{m+1})
	=e^{-i\tau \check{V}_N} \phi(\bx, t_m)
	=e^{-i\tau V}\calP_N \phi(\bx, t_m),
\end{align*}
where Fourier coefficient vector $\hat{\Psi}(t_m)=(\hat{\psi}_{\blam_1}(t_m),\hat{\psi}_{\blam_2}(t_m),\cdots, \hat{\psi}_{\blam_D}(t_m))$ satisfies
\begin{align}
	\hat{\Psi}(t_m)=e^{-i\tau \calV}\hat{\Phi}(t_m)
	\label{eq:QSM-OS2step2}
\end{align}
with
\begin{align*}
\calV=\begin{pmatrix}
\hat{V}_{\blam_1-\blam_1} & \hat{V}_{\blam_1-\blam_2} & \cdots &\hat{V}_{\blam_1-\blam_D}\\
\hat{V}_{\blam_2-\blam_1} & \hat{V}_{\blam_2-\blam_2} & \cdots &\hat{V}_{\blam_2-\blam_D}\\
\vdots & \vdots & &\vdots\\
\hat{V}_{\blam_D-\blam_1} & \hat{V}_{\blam_D-\blam_2} & \cdots &\hat{V}_{\blam_D-\blam_D}
\end{pmatrix},
\end{align*}
and $\hat{\Phi}(t_m)=(\hat{\phi}_{\blam_1}(t_m),\hat{\phi}_{\blam_2}(t_m),\cdots, \hat{\phi}_{\blam_D}(t_m))$.
Note that $\hat{V}_{\blam_j-\blam_{\ell}}$ means the Fourier coefficient of $V$ on the Fourier exponent $\blam=\blam_j-\blam_{\ell}$.
%Denote $\hat{\psi}_{\blam}(t_m)=\hat{\phi}_{\blam}(t_m)
%e^{-i\tau (\sum_{\bm{\beta}\in\bLam_N(u)} \hat{V}_{\bm{\beta}-\blam})}$.

$\bullet$ \textbf{QSM-OS2-Step 3.} For $t\in [t_m, t_m+\tau/2]$, we still consider \eqref{eq:subode1} but with initial value $\psi(\bx, t_m)$, and have
\begin{align*}
u(\bx,t_{m+1})\approx e^{-\frac{i}{2}\tau \check{A}_N} \psi(\bx, t_m).
\end{align*}

Consequently, the fully discrete scheme of QSM-OS2 can be written as
\begin{align}
	u_N^{m+1}=e^{-\frac{i}{2}\tau \check{A}_N} e^{-i\tau \check{V}_N} e^{-\frac{i}{2}\tau \check{A}_N}u_N^{m}
	=\check{\calS}^{m+1}_Nu^0\approx u(\cdot, t_{m+1}),~~
	0\leq m\leq M-1,
	\label{eq:splitting-QSM}
\end{align}
where
\begin{align*}
	\check{\calS}_N=e^{-\frac{i}{2}\tau \check{A}_N} e^{-i\tau \check{V}_N} e^{-\frac{i}{2}\tau \check{A}_N}
	=e^{-\frac{i}{2}\tau A}\calP_N e^{-i\tau V}\calP_N e^{-\frac{i}{2}\tau A} \calP_N.
	%	\label{eq:fulldiscreteQSM}
\end{align*}
Note that the idenitity $\calP_Ne^{-\frac{i}{2}\tau A}\calP_Nu=e^{-\frac{i}{2}\tau A}\calP_Nu$ holds.

\subsubsection{PM-OS2 method}
From $t_m$ to $t_{m+1}$, the PM-OS2 contains three steps.

$\bullet$ \textbf{PM-OS2-Step 1.} For $t\in [t_m, t_m+\tau/2]$, similar to QSM-OS2-Step 1, we have
\begin{align*}
	u(\bx,t_m+\tau/2)
	\approx \phi(\bx,t_m)=
	e^{-\frac{i}{2}\tau A_N} u(\bx, t_m)
	=\sum_{\blam\in\sigma_N(u)} \tu_{\blam}(t_m)e^{-\frac{i}{2}\tau \Vert \blam\Vert^2}e^{i\blam\cdot \bx}.
\end{align*}
Then, denote $\tilde{\phi}_{\bk}(t_m)=\tu_{\blam}(t_m)e^{-\frac{i}{2}\tau \Vert \blam\Vert^2}$ with $\blam=\bP\bk$.

$\bullet$ \textbf{PM-OS2-Step 2.} Applying FFT yields
\begin{align*}
	I_N\phi_p(\by_j,t_m)=\sum_{\bk\in K^n_N} \tilde{\phi}_{\bk}(t_m)e^{i\bk\cdot\by_j},
\end{align*}
where the grid points $\by_j\in \bbT^n_N$.
For $t\in [t_m, t_{m+1}]$, we have
\begin{align*}
	\psi_p(\by,t_{m})=e^{-i\tau V_p} I_N \phi_p(\by, t_m),
\end{align*}
where $V_p$ is the parent function of $V$.
Using FFT again, we have $\tilde{\psi}_{\blam}(t_m)=\langle \psi_p(\by_j,t_{m}), e^{i \bk\cdot \by_j} \rangle_N$ with $\blam=\bP\bk$.
Consequently, we can obtain 
\begin{align*}
	\psi(\bx,t_{m})
	=e^{-i\tau V_N}\phi(\bx, t_m)
	=e^{-i\tau V} I_N \phi(\bx, t_m).
\end{align*}

$\bullet$ \textbf{PM-OS2-Step 3.} For $t\in [t_m, t_m+\tau/2]$, similar to PM-OS2-Step 1, we have
\begin{align*}
	u(\bx,t_{m+1})\approx
	u_N^{m+1}=e^{-\frac{i}{2}\tau A_N} \psi(\bx, t_m)
	=\sum_{\blam\in\sigma_N(u)} \tilde{\psi}_{\blam}(t_m)
	e^{-\frac{i}{2}\tau \Vert\blam \Vert^2}
	e^{i\blam\cdot \bx}.
\end{align*}

%{\bf Numerical realization:}

Therefore, we can  write the fully discrete scheme of PM-OS2 as
\begin{align}
	u_N^{m+1}=e^{-\frac{i}{2}\tau A_N} e^{-i\tau V_N} e^{-\frac{i}{2}\tau A_N}u_N^m
	=\calS^{m+1}_Nu^0\approx u(\cdot, t_{m+1}),~~
	0\leq m\leq M-1,
	\label{eq:splitting-PM}
\end{align}
where
\begin{align*}
	\calS_N=e^{-\frac{i}{2}\tau A_N} e^{-i\tau V_N} e^{-\frac{i}{2}\tau A_N}
	=e^{-\frac{i}{2}\tau A}I_N e^{-i\tau V}I_N e^{-\frac{i}{2}\tau A} I_N.
	%		\label{eq:fulldiscrete}
\end{align*}

\subsubsection{Computational complexity analysis}

Here we analyze the computational complexity of each time step for QSM-OS2 and PM-OS2, respectively. In the implementation of QSM-OS2-Step 1 and 3, the QSM requires $D$ multiplication operators to solve \eqref{eq:subode1}. In QSM-OS2-Step 2, we compute the finite terms of the Taylor expansion for $e^{-i\tau \calV}$, \textit{i.e.,} 
\begin{align}
	e^{-i\tau \calV}\approx\sum_{j=0}^{k}\frac{(-i\tau \calV)^j}{j!}.
	\label{eq:taylorexp}
\end{align}
Therefore, there require $2(k-1)D^3+D^2$ operators to compute \eqref{eq:taylorexp} and $2D^2-D$ operators to compute \eqref{eq:QSM-OS2step2}. The computational complexity of QSM-OS2 in solving TQSE \eqref{eq:QSE-re} is $O(D^3)$.
%{\color{red}$m(2D^2+D)+2(k-1)D^3+D^2$}. 

In the implementation of PM-OS2, there require $D$ multiplication operators in PM-OS2-Step 1 and 3.
However, in PM-OS2-Step 2, the availability of FFT allows us to compute the convolutions economically in physical space as dot product, only requiring $O(D \log D)$ operators. Therefore, the computational complexity of PM-OS2 in solving TQSE \eqref{eq:QSE-re} is $O(D \log D)$.

\section{Theoretical analysis}
\label{sec:converanaly}
%Taking PM-OS2 as an example, we give its convergence analysis.
%The QSM is an extension of Fourier spectral method.
In this section, we present the convergence analysis of QSM-OS2 and PM-OS2. 

\subsection{Preliminaries}
\label{sec:Pre}
In this subsection, we will introduce some Hilbert spaces on $\bbT^n$ and quasiperiodic Hilbert spaces on $\bbR^n$. 

\begin{itemize}
\item {\bf $L^2(\bbT^n)$ space: }
    $
	L^2(\bbT^n)=\Big\{U(\by): \frac{1}{|\bbT^n|}\int_{\bbT^n}\vert U\vert^2\,d\by < +\infty\Big\},
$
equipped with inner product
\begin{align*}
	(U_1, U_2)_{L^2(\bbT^n)}=\frac{1}{|\bbT^n|}\int_{\bbT^n}U_1\overline{U}_2\,d\by.
\end{align*}

\item {\bf $H^\alpha(\bbT^n)$ space:} for any integer $\alpha\geq 0$, the $\alpha$-derivative Hilbert space on $\bbT^n$ is
\begin{align*} 
	H^\alpha(\bbT^n)=\{U\in L^2(\bbT^n): \Vert U\Vert_{\alpha}< +\infty \},
\end{align*}
where  
$
\Vert U \Vert_\alpha=\Big(\sum_{\bk\in\bbZ^n}(1+\vert \bk \vert^{2})^{\alpha}
\vert \hU_{\bk}\vert^2 \Big)^{1/2},~\hU_{\bk}=(U,e^{i\bk\cdot\by})_{L^2(\bbT^n)}.
$
The semi-norm of $H^\alpha(\bbT^n)$ can be defined as
$
\vert U \vert_\alpha=\Big(\sum_{\bk\in\bbZ^n}\vert \bk \vert^{2\alpha} \vert \hU_{\bk}\vert^2 \Big)^{1/2}.
$

\item {\bf $X_{\alpha}$ space:} for any $U\in L^2(\bbT^n)$ and $\alpha\in\bbR$, the Fourier series expansion is
\begin{align*}
	U(\by)=\sum_{\bk\in\bbZ^n} \hU_{\bk}e^{i\bk\cdot \by},
\end{align*}
the linear operator $(-\Delta)^\alpha$ is given by
\begin{align*}
	(-\Delta)^\alpha U
	=\sum_{\bk\in\bbZ^n} \Vert \bk \Vert^{2\alpha} \hU_{\bk}  e^{i\bk\cdot \by},
\end{align*}
and
\begin{align*}
	X_{\alpha}=
	\Big \{U(\by)=\sum_{\bk\in\bbZ^n} \hU_{\bk}e^{i\bk\cdot \by}\in L^2(\bbT^n):
	\Vert  (-\Delta)^\alpha U \Vert^2
	=\sum_{\bk\in\bbZ^n} \vert \hU_{\bk}\vert^2\cdot
	\Vert \bk \Vert^{4\alpha} <\infty \Big \}.
\end{align*}
The set $X_{\alpha}$ forms a Hilbert space with inner product 
\begin{align*}
	(U,W)_{X_{\alpha}}=(U,W)_{L^2(\bbT^n)}
	+((-\Delta)^\alpha U,(-\Delta)^\alpha W)_{L^2(\bbT^n)},
\end{align*}
and
\begin{align*}
	\Vert U\Vert^2_{X_{\alpha}}=\sum_{\bk\in\bbZ^n} 
	(1+\Vert \bk \Vert^{4\alpha}) \vert \hU_{\bk} \vert^2.
\end{align*}
When $\alpha=0$, $\Vert U\Vert_{X_0}=\Vert U \Vert_0=\Vert U\Vert$.

\item  $\QP(\bbR^d)$ {\bf space}: let $\QP(\bbR^d)$ be the space of all $d$-dimensional quasiperiodic functions.
%Now we introduce the Hilbert spaces of the quasiperiodic function.
%{\color{blue}
	%Let $\Tri(\bbR^d)$ denote the complex vector space of all trigonometric polynomials in $\bbR^d$, that is $P\in \Tri(\bbR^d) $ if and only if there exist $c_1,\cdots, c_w\in \bbC$ and
	%$\blam_1,\cdots,\blam_w\in\bbR^d$ such that
	%\begin{align*}
	%	P(\bx)=\sum^{w}_{j=1} c_j e^{i\blam_j\cdot \bx},
	%\end{align*}
	%where $\blam_1,\cdots,\blam_w\in\bbR^d$ and $w$ is finite.
	%Obviously, $\Tri(\bbR^d)\subset \QP(\bbR^d)$.}

\item {\bf $L^q_{QP}(\bbR^d)$ space}: for any fixed $q\in [1,\infty)$, denote 
\begin{align*}
	L^q_{QP}(\bbR^d)=\Big \{v(\bx)\in\QP(\bbR^d):~\Vert v \Vert^q_{q}
	=\bbint \vert v(\bx) \vert^q\,d\bx  < \infty \Big \},
\end{align*}
and
\begin{align*}
	L^{\infty}_{QP}(\bbR^d)=\{v(\bx)\in\QP(\bbR^d):~\Vert v \Vert_{\infty}=\sup_{\bx\in \bbR^d}\vert v(\bx) \vert < \infty \},
\end{align*}
the inner product $(\cdot, \cdot)_{L_{QP}^2(\bbR^d)}$ 
\begin{align*}
(v, w)_{L^2_{QP}(\bbR^d)}=\bbint  v(\bx)\bar{w}(\bx)\,d\bx.
\end{align*}
By the Parseval identity \eqref{eq:parseval},  we have
\begin{align*}
	\Vert v \Vert_{L^2_{QP}(\bbR^d)}^2=\sum_{\blam\in\sigma(v)}
	\vert \hv_{\blam} \vert^2.
\end{align*}

% The spaces $\bbC^k_{QP}(\bbR^d)$ are naturally defined as the spaces whose elements $v$ are such that $\partial^{\alpha}v\in \bbC^0_{QP}(\bbR^d)=L^{\infty}_{QP}(\bbR^d)$ for all $\alpha\in J_k$, and $\bbC^{\infty}_{QP}(\bbR^d)=\bigcap^{\infty}_{k=0}\bbC^k_{QP}(\bbR^d)$.

\item {\bf $\calC^{\alpha}_{QP}(\bbR^d)$ space:} the space $\calC^{\alpha}_{QP}(\bbR^d)$ consists of quasiperiodic functions with continuous derivatives up to order $\alpha$ on $\bbR^d$.
The $\calC^{\alpha}_{QP}$-norm of $v\in \calC^{\alpha}_{QP}(\bbR^d)$ is defined by
\begin{align*}
	\Vert v \Vert_{\calC^{\alpha}_{QP}} =\sum_{\vert \bmm\vert \leq \alpha} \sup_{\bx \in \bbR^d}  \vert \partial_{\bx}^{\bmm}v \vert.  
\end{align*}

\item {\bf $H^{\alpha}_{QP}(\bbR^d)$ space:} for any $\alpha\in\bbN_0$, the Hilbert space $H^{\alpha}_{QP}(\bbR^d)$ comprises all quasiperiodic functions with partial derivatives order $\alpha \geq 1$. For $v,w\in H^{\alpha}_{QP}(\bbR^d)$, the inner product $(\cdot, \cdot)_{H^{\alpha}_{QP}(\bbR^d)}$ is
\begin{align*}
	(v, w)_{H^{\alpha}_{QP}(\bbR^d)}=(v, w)_{L^2_{QP}(\bbR^d)}+ \sum_{\vert
		\bmm\vert=\alpha}(\partial^{\bmm}_{\bx} v, \partial^{\bmm}_{\bx} w)_{L^2_{QP}(\bbR^d)}.
\end{align*}
%\begin{thm}[Parseval's equality]
%	\label{thm:parseval}
%	If $f(\bx)\in \QP(\bbR^d)$ such that $f(\bx)\sim \sum_{j=1}^{\infty}c_{\blam_j}(f) e^{i\blam_j\cdot \bx}$,
%	then Parseval's equality
%	\begin{align}
	%		\sum_{j=1}^{\infty}\vert c_{\blam_j}(f)\vert^2
	%		=\lim_{L\rightarrow\infty} \frac{1}{\vert \Omega_L\vert} 
	%		\int_{\Omega_L} \vert f(\bx)\vert^2\,d\bx
	%		\label{eq:parseval}
	%	\end{align}
%	is true.
%\end{thm}
The corresponding norm is
\begin{align*}
	\Vert v \Vert_{H^{\alpha}_{QP}(\bbR^d)}^2=\sum_{\blam\in\sigma(v)}(1+\vert \blam\vert^2)^{\alpha}\vert \hv_{\blam} \vert^2.
\end{align*}
In particular, for $\alpha=0$, $H^0_{QP}(\bbR^d)=L^2_{QP}(\bbR^d)$.
To simplify notation, we denote  $(\cdot, \cdot)=(\cdot, \cdot)_{L^2_{QP}(\bbR^d)}
$ and $(\cdot, \cdot)_{\alpha}=(\cdot, \cdot)_{H^{\alpha}_{QP}(\bbR^d)}$.
%Correspondingly, $\Vert \cdot \Vert=\Vert \cdot\Vert_{L_{QP}^2(\bbR^d)}$ and $\Vert \cdot \Vert_{\alpha}=\Vert \cdot \Vert_{H_{QP}^{\alpha}(\bbR^d)}$. 
The embedding theorem of $H^{\alpha}_{QP}(\bbR^d)$ can be found in \cite{Iannacci1998embedding}.
\end{itemize}

% For any $v\in\QP(\bbR^d)$, assume that the Fourier series expansion is
% \begin{align*}
	% 	v(\bx)=\sum_{\blam\in\sigma(v)} \hv_{\blam}e^{i\blam\cdot \bx}.
	% \end{align*}
% For $\alpha\in\bbR$, the linear operator $(-\Delta)^\alpha$ is given by
% \begin{align*}
	% 	(-\Delta)^\alpha v
	% 	=\sum_{\blam\in\sigma(v)} \hv_{\blam} \Vert \blam \Vert^{2\alpha} e^{i\blam\cdot \bx},
	% \end{align*}
% and
% \begin{align*}
	% 	X_{\alpha}=
	% 	\Big \{v(\bx)=\sum_{\blam\in\sigma(v)} \hv_{\blam}e^{i\blam\cdot \bx}\in\QP(\bbR^d):
	% 	\Vert  (-\Delta)^\alpha v \Vert^2
	% 	=\sum_{\blam\in\sigma(v)} \vert \hv_{\blam}\vert^2\cdot
	% 	\Vert \blam \Vert^{4\alpha} <\infty \Big \}.
	% \end{align*}
% The domain $X_{\alpha}$ forms a Hilbert space with inner product 
% \begin{align*}
	% 	(v,w)_{X_{\alpha}}=(v,w)+((-\Delta)^\alpha v,(-\Delta)^\alpha w),
	% \end{align*}
% and
% \begin{align*}
	% 	\Vert v\Vert_{X_{\alpha}}=\sum_{\blam\in\sigma(v)} 
	% 	(1+\Vert \blam \Vert^{4\alpha}) \vert \hv_{\blam} \vert^2.
	% \end{align*}
% Obviously, when $\alpha=0$, $\Vert v\Vert_{X_0}=\Vert v\Vert$.

A recent study has revealed an important relationship between quasiperiodic function and its parent function, see Lemma \ref{lem:object}.
\begin{lemma}[\cite{Jiang2022Numerical}]
	\label{lem:object}
	Consider a quasiperiodic function $v(\bx)=v_p(\bP^{T}\bx)$ where $v_p(\by)$ is its parent function and $\bP\in \bbP^{d\times n}$. Let the quasiperiodic Fourier coefficient $\hv_{\blam}=(v, e^{i\blam\cdot \bx})$ and the periodic Fourier coefficient $\hv_p(\bk)=(v_p, e^{i\bk\cdot \by})_{L^2(\bbT^n)}$. Then, we have
	\begin{align*}
		\hv_{\blam}= \hv_p(\bk),
	\end{align*}
	where $\blam=\bP\bk$, $\bk\in\bbZ^n$. 
	% 	\begin{align}
		% 		\hu_{\blam}=\bbint e^{-i\blam\cdot \bx}u(\bx)\,d\bx,
		% 		\label{eq:defak}
		% 	\end{align}
	% 	and
	% 	\begin{align}
		% 		\hu_p(\bk)=\frac{1}{\vert \bbT^n\vert}\int_{\bbT^n}e^{-i\bk\cdot\by}u_p(\by)\,d\by.
		% 		\label{eq:defbk}
		% 	\end{align}
	%	with $|\bbT^n|$ is the measure of the $n$-dimensional tours.
\end{lemma}
Based on Lemma \ref{lem:object}, we prove the norm inequality between quasiperiodic function and its parent function.
\begin{thm}
\label{lem:normineq}
	%	[Lemma 1 (ii) \cite{Thalhammer2012Convergence}]
	For any $v\in \QP(\bbR^d)$, and $v_p$ is the corresponding $n$-dimensional parent function. Assume that $\alpha\geq s/2> n/4$ and $v_p\in X_{\alpha}$. Then, the bound
	\begin{align}
		\Vert v\Vert_{L_{QP}^{\infty}(\bbR^d)}\leq C\, \Vert v_p\Vert_{s} \leq C\, \Vert v_p\Vert_{X_{\alpha}}
		\label{eq:inftyembedX}
	\end{align} 
	is valid and the following estimate
	\begin{align*}
		\Vert wv \Vert_{L_{QP}^2(\bbR^d)} \leq C \, \Vert w \Vert_{L_{QP}^2(\bbR^d)} \cdot\Vert v_p\Vert_{X_\alpha},
		~~w\in L_{QP}^{2}(\bbR^d),~~v_p\in X_{\alpha}
	\end{align*}
	holds where $C$ is a constant.
\end{thm}

\begin{proof} 
	%Firstly, we prove that the inequality $\Vert v\Vert_{L_{QP}^{\infty}(\bbR^d)}\leq C\, \Vert v_p\Vert_{\delta}$ is true. 
	Applying Lemma \ref{lem:object} and Cauchy-Schwarz inequality, we have
	\begin{align*}
		\sum_{j=1}^{\infty}\vert \hv_{\blam_j}\vert
		&=\sum_{j=1}^{\infty}\Big [\vert \hv_p(\bk_j)\vert\, (1+\vert \bk_j\vert^2)^{s/2}\, (1+\vert \bk_j\vert^2)^{-s/2} \Big]\\
		&\leq \sum_{j=1}^{\infty}\Big [\vert \hv_p(\bk_j)\vert^2\, (1+\vert \bk_j\vert^2)^{s}\Big]\,\sum_{j=1}^{\infty} (1+\vert \bk_j\vert^2)^{-s}\\
		&=C\Vert v_p\Vert_s.
	\end{align*}
	The last estimate holds since $s> n/2$ implies the series $\sum_{j=1}^{\infty} (1+\vert \bk_j\vert^2)^{-s}$ converges.
	%Note that $v_p\in X_{\alpha}$ implies $v\in L_{QP}^1(\bbR^d)$. 
	When the Fourier series of $v$ is absolutely convergent, then it uniformly convergent to the quasiperiodic function $v$ (Theorem 1.20 in \cite{Corduneanu1989almost}), such that
	$\Vert v\Vert_{L_{QP}^{\infty}(\bbR^d)}\leq \sum_{j=1}^{\infty}\vert \hv_{\blam_j}\vert$. 
	Therefore, we have
	$\Vert v\Vert_{L_{QP}^{\infty}(\bbR^d)}\leq C\, \Vert v_p\Vert_{s}$, where $C$ is a positive constant. 
	Using the definition of $\Vert \cdot\Vert_{s}$ and $\Vert \cdot \Vert_{X_\alpha}$, we can obtain
	\begin{align*}
		\Vert v_p\Vert_{s} \leq C\,\Vert v_p\Vert_{X_{2s}},
	\end{align*}
	and for $0\leq \alpha_1\leq \alpha_2$,
	\begin{align*}
		\Vert v_p\Vert_{X_{\alpha_1}} \leq \Vert v_p\Vert_{X_{\alpha_2}}.
	\end{align*}
	Combining $\alpha\geq s/2$, it follows that 
	$\Vert v_p\Vert_{s}\leq C\,\Vert v_p\Vert_{X_{\alpha}}$. Therefore, the inequality \eqref{eq:inftyembedX} holds.
Further, for any $w\in L_{QP}^{2}(\bbR^d)$, we can obtain
\begin{align*}
\Vert vw \Vert_{L_{QP}^2(\bbR^d)} \leq \Vert w\Vert_{L_{QP}^2(\bbR^d)} \cdot
\Vert v \Vert_{L_{QP}^{\infty}(\bbR^d)}
		\leq C\, \Vert w\Vert_{L_{QP}^2(\bbR^d)} \cdot \Vert v_p \Vert_{X_\alpha}.
	\end{align*}
\end{proof}

% \begin{remark}
	% \jk{REWORD: what's your innovation ??}
	
	% Iannacci \textit{et al.} \cite{Iannacci1998embedding} have derived an embedding theorem of quasiperiodic function space $H^{\alpha}_{QP}(\bbR^d)$ when the Fourier exponents satisfy the below assumption
	% \begin{align*}
		% \sum_{j=1}^{\infty}\frac{1}{\vert\blam_j\vert^\gamma}< \infty, ~~\forall~\gamma >\beta\geq 0.
		% \end{align*}
	% Here, $\beta$ is the exponent relevant to the regularity of the embeddings. With these embedding theorems, we can also analyze the norm inequality of quasiperiodic function. 
	% \end{remark}

%  \begin{remark}
	%  The small divisor problem is always unavoidable in the analysis of many quasiperiodic systems. In the following numerical analysis of TQSE, inspired by the relationship between quasiperiodic functions and corresponding parent functions, this problem can be effectively avoided. 
	%  \end{remark}

\subsection{Convergence analysis}

This subsection presents the convergence analysis of full discretization scheme \eqref{eq:splitting-PM}.

\subsubsection{Main theorem}

Let $u^j_p$ denote the parent function of $u^j$, where $u^j$ is the approximate solution at $t=t_j$ exactly satisfying the scheme \eqref{eq:splitting}. Below, we give the main theorem of the error analysis of PM-OS2.
\begin{thm}
	\label{thm:PMerror}
	Let $u(\cdot,t_m)$ and $u^m_N$ be the solutions of problems \eqref{eq:QSE-re} and \eqref{eq:splitting-PM} at $t_m$, respectively. Then under the conditions
	
	(i) The potential $V$ is a $\calC^1$-smooth function and $\Vert V_p\Vert_{X_\alpha}\leq C_V$, $\alpha\geq s/2> n/4$;
	
	(ii) The parent function $u^j_p\in H^{\alpha}(\bbT^n),~0\leq j\leq m$;\\
	the global error bound of PM-OS2 \eqref{eq:splitting-PM} is	
	\begin{align*}
		\Vert u^m_N- u(\cdot,t_m)\Vert_{L_{QP}^2(\bbR^d)}\leq C(\tau^2 + N^{-\alpha}).
	\end{align*}
	The constant $C>0$ depends on the bounds $C_V$, $\sup\{\Vert u(\cdot, t) \Vert_{L_{QP}^2(\bbR^d)}: 0\leq t \leq T \}$ and $\max\{\vert u^j_p\vert_{\alpha}: 0\leq j\leq m \}$.
\end{thm}

The detailed proof of Theorem \ref{thm:PMerror} is given at the end of this section.

\subsubsection{Some results}

Before proving the main conclusion, we present some required results.

\begin{lemma}
\label{thm:selfadjoint}
The operator $A$ defined in \eqref{eq:QSE-re} is self-adjoint in the inner product $L^2_{QP}$.
\end{lemma}
\begin{proof}
When $v,w\in \QP(\bbR^d)$ and are differentiable, applying the Green identity, we have
\begin{align*}
(Av, w)-(v, Aw)
&=-(\Delta v, w)+(v, \Delta w)\\
&=-\bbint\Delta v\cdot w\,d\bx
+\bbint v\cdot \Delta w\,d\bx\\
&=-\lim_{L\rightarrow\infty} \frac{1}{\vert \Omega_L\vert} \int_{\partial\Omega_L} 
\frac{\partial v}{\partial \bn}\cdot w\,dS
+\lim_{L\rightarrow\infty} \frac{1}{\vert \Omega_L\vert} \int_{\partial\Omega_L} 
v\cdot \frac{\partial w}{\partial \bn}\,dS,
\end{align*}
	where $\bn$ is the outward normal of $\partial\Omega_L$. Since the quasiperiodic functions $v, w$ are bounded, 
	\text{i.e.,} $\sup_{\bx\in \bbR^d} \{v(\bx), w(\bx)\}\leq C<\infty$ , it follows that
	\begin{align*}
		\Big \vert\lim_{L\rightarrow\infty} \frac{1}{\vert \Omega_L\vert} \int_{\partial\Omega_L} 
		\frac{\partial v}{\partial \bn}\cdot w\,dS \Big \vert	
		\leq \lim_{L\rightarrow\infty} \frac{C^22d(2L)^{d-1}}{(2L)^d}=0,~~
		\Big \vert\lim_{L\rightarrow\infty} \frac{1}{\vert \Omega_L\vert} \int_{\partial\Omega_L} 
		v \cdot \frac{\partial w}{\partial \bn}\,dS \Big \vert	
		=0.   
	\end{align*}
	Therefore, $(-\Delta v,w)=(v,-\Delta w)$ is true.
	The proof of this theorem is completed.
\end{proof}

We recall some basic definitions that occur in the following.

\begin{defy}
	The operator $B$ is unitary if $BB^*=B^*B=\calI$, where $\calI$ is the identity operator and $B^*$ is adjoint operator of $B$.
\end{defy}

\begin{defy}
	A one-parametric family $\Psi(t), t\in\bbR$, of linear operators on a Banach space $X$ is a group of operators on $X$ if $\Psi(0)=\calI$ and $\Psi(t+s)=\Psi(t)\Psi(s), 
	s\in\bbR$. 
	Moreover, $\Psi(t)$ is $C_0$-group if $\lim\limits_{t\rightarrow 0} \Psi(t) x=x,$ for all $x\in X$.
\end{defy}

\begin{lemma}[Stone Theorem \cite{Stone1930linear}]
	\label{thm:stone}
	Let $H$ be a Hilbert space, and $\Psi(t)$ be a one-parameter family of linear operators on $H$ for $t\in \bbR$.
	
	(i) If $\Psi(t)$ is $C_0$-group of unitary operators on $H$, then there exists a unique self-adjoint operator $A$ such that $\Psi(t)=e^{-itA}$, where $-iA$ is the generator.
	
	(ii) Conversely, let $A$ be a self-adjoint operator on $H$. Then $\Psi(t)=e^{-itA}$ is a $C_0$-group of unitary operator with the generator $-iA$.
\end{lemma}

Lemma \ref{thm:selfadjoint} and Lemma \ref{thm:stone} demonstrate $\Vert e^{-itA}\Vert_{L_{QP}^2(\bbR^d)}=1$.
Next, we will give an error bound of OS2 method \eqref{eq:splitting} for TQSE \eqref{eq:QSE-re}. For operators $A$ and $B$, we denote the commutators $[A,B]=AB-BA$ and $[A,[A,B]]=A^2B-2ABA+BA^2$.

\begin{lemma}[Theorem 2.1, \cite{Jahnke2000error}]
	\label{lem:localerror}
	Suppose that the operators $\Phi$ and $\Psi$ are self-adjoint on the Hilbert space $L^2_{QP}(\bbR^d)$. 
	
	(1) Under the following conditions
	
	(i) $\Psi$ is bounded, \text{i.e.,} $\Vert \Psi v\Vert_{L_{QP}^2(\bbR^d)} \leq \beta\Vert v\Vert_{L_{QP}^2(\bbR^d)}$, where $\beta$ is a constant and $v\in L^2_{QP}(\bbR^d)$.
	
	(ii) Commutator bound
	$
	\Vert [\Phi, \Psi] v\Vert_{L_{QP}^2(\bbR^d)} \leq c_1 \Vert v\Vert_{H^1_{QP}}, ~\forall ~v\in H^1_{QP}(\bbR^d).
	$
	
	We have 
	\begin{align*}
		\Vert e^{-\frac{i}{2}\tau \Psi} e^{-i\tau \Phi} e^{-\frac{i}{2}\tau \Psi}v
		-e^{-i\tau (\Phi +\Psi)}v\Vert_{L_{QP}^2(\bbR^d)} \leq C_1 \tau^2 \Vert v\Vert_{L_{QP}^2(\bbR^d)}, 
	\end{align*}
	where $C_1$ depends only on $c_1$ and $\beta$.
	
	(2) Under the conditions in (1) and additionally
	
	(iii) Commutator bound
	$
	\Vert[\Phi, [\Phi, \Psi] ] v\Vert_{L_{QP}^2(\bbR^d)} \leq c_2 \Vert v\Vert_{L_{QP}^2(\bbR^d)},
	~\forall~v\in L^2_{QP}(\bbR^d).
	$ 
	
	We have
	\begin{align*}
		\Vert e^{-\frac{i}{2}\tau \Psi} e^{-i\tau \Phi} e^{-\frac{i}{2}\tau \Psi}v
		-e^{-i\tau (\Phi +\Psi)}v\Vert_{L_{QP}^2(\bbR^d)} \leq C_2 \tau^3 \Vert v\Vert_{L_{QP}^2(\bbR^d)}, 
	\end{align*}
	where $C_2$ depends only on $c_1$, $c_2$ and $\beta$.
\end{lemma}

\begin{thm}
\label{thm:timeerror}
For a quasiperiodic function $V$, the following conclusions are valid.

(i) If $V\in \calC^2_{QP}(\bbR^d)$, then
$ \Vert \calS v-\calT v\Vert_{L_{QP}^2(\bbR^d)}
\leq C \tau^2 \Vert v \Vert_{L_{QP}^2(\bbR^d)} $. 

(ii) If $V\in \calC^1_{QP}(\bbR^d)$, we have
$\Vert \calS v-\calT v\Vert_{L_{QP}^2(\bbR^d)} \leq C \tau^3 \Vert v \Vert_{L_{QP}^2(\bbR^d)} $.
\end{thm}

\begin{proof}
	%Obviously, $V$ is bounded.
	According to Lemma \ref{thm:selfadjoint}, the operator $A=-\Delta$ is self-adjoint in the sense of $L^2_{QP}$-norm. Next, we verify that the operators $A$ and $V$ satisfy commutator bounds.
	From the equation
\begin{align*}
\Delta(\phi\psi)=(\Delta\phi) \psi+2(\nabla \phi) (\nabla \psi)+ \phi(\Delta\psi),
\end{align*}
we have
\begin{align*}
\Vert [\Delta, V]\phi\Vert_{L_{QP}^2(\bbR^d)}
\leq \Vert(\Delta V)\phi \Vert_{L_{QP}^2(\bbR^d)} 
+2\Vert (\nabla V)\cdot (\nabla \phi)\Vert_{L_{QP}^2(\bbR^d)}
\leq C(\Vert V\Vert_{\calC_{QP}^2}, \Vert \phi\Vert_{1}),  \end{align*}
	%and
	%\begin{align*}
	%	\Vert[\Delta, [\Delta, V]]\phi\Vert
	%	\leq 8 \Vert(\Delta V)(\Delta \phi) \Vert+\Vert(\Delta^2 V) \phi\Vert 
	%	+4\Vert (\nabla^3 V)\cdot (\nabla \phi)\Vert
	%	\leq C(\Vert V\Vert_{C_{QP}^4}, \Vert \phi\Vert_{H^2}).  
	%\end{align*}
and
\begin{align*}
\Vert[V, [V, \Delta]]\phi\Vert_{L_{QP}^2(\bbR^d)}
=\Vert 2 (\nabla V)(\nabla V)\phi\Vert_{L_{QP}^2(\bbR^d)}
\leq C(\Vert V\Vert_{\calC_{QP}^1}, \Vert \phi\Vert).  
\end{align*}
Combining with the conclusions in Lemma \ref{lem:localerror}, these conclusions in this theorem are easy to prove.
\end{proof}

%{\color{red}
	%\begin{theorem}
	%For a $C^5$-smooth potential $V$, the error of the Strang splitting method at time $t_n=n\tau$ is bounded by
	%\begin{align*}
	%\Vert u^n-u(t_n) \Vert \leq C\tau^2 t_n \max_{o\leq j\leq n}\Vert u(t_j) \Vert_2,
	%\end{align*}
	%{\color{red}where $C$ depends only on ?.}
	%\end{theorem}
	%
	%\begin{proof}
	%Since...
	%\end{proof}
	%}

%As standard, for any integer number $1\leq p\leq \infty$, we denote by $L^p( \Omega)$ the Lebesgue space of complex-valued functions on $\Omega$. The Sobolev space $W^{m,p}(\Omega)$ comprises all functions with partial derivatives up to order $m\leq 1$ contained in $L^p( \Omega)$.
%In our analysis, we default to $V_p$ as the parent function of $V$.
To prove the difference between $\calS_N$ and $I_N\calS$, we first introduce the following lemmas.

\begin{lemma}
	\label{lem:defectAIN-INA}
	It holds
	\begin{align*}
		\Vert (I_N e^{-\frac{i}{2}\tau A}-e^{-\frac{i}{2}\tau A}I_N )v\Vert_{L_{QP}^2(\bbR^d)}
		\leq \int^\tau_{0} \Vert[A, I_N-\calI] e^{-\frac{i}{2}t A} v \Vert_{L_{QP}^2(\bbR^d)} \, dt.
	\end{align*}
\end{lemma}

\begin{proof}
	Consider the following two initial value problems
	\begin{align*}
		\begin{cases}
			\dfrac{d}{d\tau}u(\tau)
			=-\dfrac{i}{2}Au(\tau),\\
			u(0)=v,
		\end{cases}
		~\mbox{and}~~~~
		\begin{cases}
			\dfrac{d}{d\tau}w(\tau)
			=-\dfrac{i}{2}Aw(\tau),\\
			w(0)=I_Nv.
		\end{cases}
	\end{align*}
	For the initial value problem
	\begin{align}
		\frac{d}{d\tau}(w-I_Nu)(\tau)
		=-\frac{i}{2}Aw(\tau)-I_N(-\frac{i}{2}A)u(\tau),
		\label{eq:lemdefectAproof} 
	\end{align}
	with initial value $(w-I_Nu)(0)=0$, we have 
	\begin{align*}
		(w-I_Nu)(\tau)=e^{-\frac{i}{2}\tau A}I_N v-I_N e^{-\frac{i}{2}\tau A}v.
	\end{align*}
	We rewrite \eqref{eq:lemdefectAproof} as
	\begin{align*}
		\frac{d}{d\tau}(w-I_Nu)(\tau)
		=-\frac{i}{2}A(w-I_Nu)(\tau)-\frac{i}{2}AI_Nu(\tau)-I_N(-\frac{i}{2}A)u(\tau).
	\end{align*}
	Applying the linear variation-of-constants formula, we obtain
	\begin{align}
		(w-I_Nu)(\tau)
		&=e^{-\frac{i}{2}A\tau}(w-I_Nu)(0)
		+\int^{\tau}_{0}e^{-\frac{i}{2}A(t-\tau)}[-\frac{i}{2}A,I_N]
		e^{-\frac{i}{2}t A} v\,dt \notag\\
		&=\int^{\tau}_{0}e^{-\frac{i}{2}A(t-\tau)}[-\frac{i}{2}A, I_N-\calI] e^{-\frac{i}{2} t A} v\,dt.
		\label{eq:prooflinearvar}
	\end{align}
	%Due to $A$ is a self-adjoint operator, applying Stone Theorem \ref{thm:stone}, we have $\Vert e^{-\frac{i}{2}A\tau} v \Vert=\Vert v \Vert $.
	For any $v\in L^2_{QP}(\bbR^d)$, applying Lemma \ref{thm:stone}, it follows that
	\begin{align}
		\Vert e^{-\frac{i}{2}\tau A} v\Vert_{L_{QP}^2(\bbR^d)} =\Vert v\Vert_{L_{QP}^2(\bbR^d)},
		\label{eq:Anorm}	
	\end{align}
	then $(e^{-\frac{i}{2}\tau A})_{\tau\in\bbR}$ on $L_{QP}^{2}(\bbR^d)$ is a unitary group. Combining with the formula \eqref{eq:prooflinearvar}, this completes the proof.
\end{proof}

\begin{lemma}
	\label{lem:defectIN-V-W}
	Assume that there exists a constant $C_V>0$ such that $\Vert V_p\Vert_{X_\alpha}\leq C_V$.
	Then we have
	\begin{align*}
		\Vert I_N e^{-i\tau V} v\Vert_{L_{QP}^2(\bbR^d)}
		\leq e^{CC_V\tau} \Vert I_N v\Vert_{L_{QP}^2(\bbR^d)},~~v\in L_{QP}^2(\bbR^d),
	\end{align*}
	where $\tau\in [0,T]$ and $C$ is a constant.
\end{lemma}

\begin{proof}
	Consider the problem
	\begin{align}
		I_N \frac{d}{d\tau}u(\tau)=I_N(-iV)u(\tau),~~\tau\in [0,T],
		\label{eq:interODE}
	\end{align}
	with initial value $I_Nu(0)=I_Nv$. We can obtain
	\begin{align*}
		I_Nu(\tau)=I_N e^{-i\tau V} v.
	\end{align*}
	For any $v_1\in L_{QP}^\infty(\bbQ_N),~v_2\in L_{QP}^2(\bbR^d)$, assume that $v_{p_1}\in X_{\alpha}$ is the parent function of $v_1$.
	Using Theorem \ref{lem:normineq}, we have
	\begin{align*}
		\Vert I_N (v_1v_2) \Vert_{L_{QP}^2(\bbR^d)} \leq \Vert v_1 \Vert_{L_{QP}^{\infty}(\bbQ_N)}\Vert I_N v_2\Vert_{L_{QP}^2(\bbR^d)}
		\leq C\,\Vert v_{p_1} \Vert_{X_{\alpha}}\Vert I_N v_2\Vert_{L_{QP}^2(\bbR^d)}.
	\end{align*}
	Integrating \eqref{eq:interODE} from $0$ to $\tau$ yields 
	\begin{align*}
		\Vert I_Nu\Vert_{L_{QP}^2(\bbR^d)}
		&\leq \Vert I_N v \Vert_{L_{QP}^2(\bbR^d)}
		+\int^\tau_{0} \Vert I_NVu(t)\Vert_{L_{QP}^2(\bbR^d)}\,dt\\
		&\leq \Vert I_Nv \Vert_{L_{QP}^2(\bbR^d)}
		+C\Vert V_p \Vert_{X_{\alpha}}\int^\tau_{0} \Vert I_Nu\Vert_{L_{QP}^2(\bbR^d)}\,dt\\
		&\leq \Vert I_N v \Vert_{L_{QP}^2(\bbR^d)}
		+CC_V\int^\tau_{0} \Vert I_Nu\Vert_{L_{QP}^2(\bbR^d)}\,dt.
	\end{align*}
	Then we can prove the theorem by applying the Gronwall inequality.
\end{proof}

The error analysis of PM in the sense of $L^2_{QP}$-norm is as follows.
\begin{lemma}[\cite{Jiang2022Numerical}]
\label{thm:pm error-2}
Suppose that $u(\bx)\in \QP(\bbR^d)$ and its parent function $u_p(\by)\in H^{\alpha}(\mathbb{T}^n)$ with $\alpha > n/2$. There exists a constant $C$, independent of $u_p$ and $N$, such that
	\begin{align*}
		\Vert I_N u-u\Vert_{L_{QP}^2(\bbR^d)} \leq CN^{-\alpha}\vert u_p \vert_\alpha.
	\end{align*}
\end{lemma}

Now, we analyze the difference between $\calS_N$ and $I_N\calS$.
Let $u_p$, $v_p$ and $w_p$ be the parent functions of quasiperiodic functions $u$, $v$ and $w$, respectively.
\begin{thm}
For $u\in \QP(\bbR^d)$, $v =\calS u$ and $w= e^{-\frac{i}{2}\tau A}u$.
Assume that $\Vert V_p \Vert_{X_\alpha} \leq C_V$ and $u_p,~v_p,~w_p\in H^{\alpha}(\mathbb{T}^n)$ with $\alpha\geq s/2> n/4$.
Then we have 
\begin{align*}
\Vert \calS_N(u)-I_N\calS(u)\Vert_{L_{QP}^2(\bbR^d)} \leq C\tau N^{-\alpha} 
(\vert v_p\vert_\alpha+\vert w_p\vert_\alpha),
\end{align*}
where $C$ is a constant.
\end{thm}

\begin{proof}
	Since
	\begin{align*}
		I_N\calS(u)-\calS_N(u)
		&=I_N e^{-\frac{i}{2}\tau A} e^{-i\tau V} e^{-\frac{i}{2}\tau A}u
		-e^{-\frac{i}{2}\tau A}I_N e^{-i\tau V}I_N e^{-\frac{i}{2}\tau A} I_Nu\\
		&=I_N e^{-\frac{i}{2}\tau A} e^{-i\tau V} e^{-\frac{i}{2}\tau A}u-
		e^{-\frac{i}{2}\tau A} I_N e^{-i\tau V} e^{-\frac{i}{2}\tau A}u\\
		&~~+e^{-\frac{i}{2}\tau A} I_N e^{-i\tau V} e^{-\frac{i}{2}\tau A}u-
		e^{-\frac{i}{2}\tau A} I_N e^{-i\tau V} I_N e^{-\frac{i}{2}\tau A}u\\
		&~~+e^{-\frac{i}{2}\tau A} I_N e^{-i\tau V} I_N e^{-\frac{i}{2}\tau A}u
		-e^{-\frac{i}{2}\tau A}I_N e^{-i\tau V}I_N e^{-\frac{i}{2}\tau A} I_Nu\\
		&=Z_1+Z_2+Z_3.
	\end{align*}
	Then we estimate the bounds of $\Vert Z_1 \Vert_{L_{QP}^2(\bbR^d)} $, $\Vert  Z_2 \Vert_{L_{QP}^2(\bbR^d)} $ and $\Vert Z_3 \Vert_{L_{QP}^2(\bbR^d)} $, respectively.
	
	(i) For $Z_1$ and using Lemma \ref{lem:defectAIN-INA}, we have
	\begin{align*}
		\Vert Z_1 \Vert_{L_{QP}^2(\bbR^d)}
		&=\Vert (I_N e^{-\frac{i}{2}\tau A}-e^{-\frac{i}{2}\tau A} I_N)
		e^{-i\tau V} e^{-\frac{i}{2}\tau A}u\Vert_{L_{QP}^2(\bbR^d)} \\
		&\leq \int^\tau_{0}  \Vert [A (I_N-\calI)- (I_N-\calI)A]
		e^{-\frac{i}{2} t A}e^{-i t V} e^{-\frac{i}{2} t A}u \Vert_{L_{QP}^2(\bbR^d)}\,dt \\
		&=\int^\tau_{0} \Vert [A (I_N-\calI)- (I_N-\calI)A]v \Vert_{L_{QP}^2(\bbR^d)}\,dt
	\end{align*}
	where $v=\calS u$.
	Since $\vert v_p \vert_\alpha< +\infty$, applying Lemma \ref{thm:pm error-2}, it follows that
	\begin{align*}
		\Vert Z_1 \Vert_{L_{QP}^2(\bbR^d)}\leq C \tau N^{-\alpha} \vert v_p \vert_\alpha.
	\end{align*}

	(ii)  For $Z_2$, using \eqref{eq:Anorm} and Lemma \ref{lem:defectIN-V-W}, we have
	\begin{align*}
		\Vert Z_2 \Vert_{L_{QP}^2(\bbR^d)}
		&=\Vert e^{-\frac{i}{2}\tau A} I_N e^{-i\tau V}( e^{-\frac{i}{2}\tau A}u
		-I_N e^{-\frac{i}{2}\tau A} u)\Vert_{L_{QP}^2(\bbR^d)}\\
		&= \Vert I_N e^{-i\tau V} (e^{-\frac{i}{2}\tau A}u
		-I_N e^{-\frac{i}{2}\tau A} u)\Vert_{L_{QP}^2(\bbR^d)}\\
		& \leq e^{CC_V\tau} \Vert I_N (e^{-\frac{i}{2}\tau A}u
		-I_N e^{-\frac{i}{2}\tau A} u)\Vert_{L_{QP}^2(\bbR^d)}=0.
	\end{align*}

	(iii) For $Z_3$, using \eqref{eq:Anorm}, Lemma \ref{lem:defectAIN-INA} and Lemma \ref{lem:defectIN-V-W}, we have
	\begin{align*}
		\Vert Z_3 \Vert_{L_{QP}^2(\bbR^d)}
		&=\Vert e^{-\frac{i}{2}\tau A} I_N e^{-i\tau V} (I_N e^{-\frac{i}{2}\tau A}u
		- I_N e^{-\frac{i}{2}\tau A}I_N u) \Vert_{L_{QP}^2(\bbR^d)}\\
		&\leq e^{CC_V\tau}\Vert I_N (I_N e^{-\frac{i}{2}\tau A}u
		- I_N e^{-\frac{i}{2}\tau A}I_N u) \Vert_{L_{QP}^2(\bbR^d)}\\
		&=e^{CC_V\tau}\Vert (I_N e^{-\frac{i}{2}\tau A}
		-e^{-\frac{i}{2}\tau A}I_N) u \Vert_{L_{QP}^2(\bbR^d)}\\
		& \leq  e^{CC_V\tau}
		\int^\tau_{0} \Vert [A, I_N-\calI] e^{-\frac{i}{2} t A} u \Vert_{L_{QP}^2(\bbR^d)}\, dt  \\
		&\leq C \tau N^{-\alpha} \vert w_p  \vert_\alpha.
	\end{align*}
\end{proof}

% \begin{lemma}
	% \label{lem:boundV}
	% Let $v\in L_{QP}^2(\bbR^d)$. Assume that there exists a constant $C_V>0$ such that $\Vert V_p\Vert_{X_\alpha}\leq C_V$. Then we have
	% 	\begin{align*}
		% 		\Vert e^{-\frac{i}{2}\tau V} v\Vert
		% 		\leq e^{CC_V\tau} \Vert v\Vert,~~\tau\in [0,T],
		% 	\end{align*}
	% where $C$ is a constant.
	% \end{lemma}

% \begin{proof}
	% 	In order to deduce the bound for $u(\tau)=e^{-\frac{i}{2}\tau V}v$, we consider the initial value problem
	% 	\begin{align*}
		% 		\frac{d}{d\tau}u(\tau)=-\frac{i}{2}Vu(\tau),~~\tau\in [0,T],~~u(0)=v.
		% 	\end{align*}
	% 	Integration and applying Theorem \ref{lem:normineq}, we have
	% 	\begin{align*}
		% 		u(\tau)=v-\frac{i}{2}\int^{\tau}_{0}V u(t)\,dt
		% 	\end{align*}
	% 	and
	% 	\begin{align*}
		% 		\Vert u(\tau)\Vert 
		% 		\leq \Vert v\Vert +C \Vert V_p \Vert_{X_\alpha} \int^{\tau}_{0}\Vert u(t)\Vert\,dt
		% 		\leq \Vert v\Vert +C C_V \int^{\tau}_{0}\Vert u(t)\Vert\,dt.
		% 	\end{align*}
	% 	According to the Gronwall-type inequality, we can obtain
	% 	\begin{align*}
		% 		\Vert u(\tau)\Vert\leq e^{C C_V\tau}\Vert v\Vert, ~~\tau\in [0,T],
		% 	\end{align*}
	% 	which implies the stated result.
	% \end{proof}

\subsubsection{The proof of Theorem \ref{thm:PMerror}}

Using the triangle inequality, we have
	\begin{align*}
		\Vert u^m_N- u(\cdot,t_m)\Vert_{L_{QP}^2(\bbR^d)}\leq \Vert u^m_N- u^m\Vert_{L_{QP}^2(\bbR^d)}+ \Vert u^m- u(\cdot, t_m)\Vert_{L_{QP}^2(\bbR^d)},
	\end{align*}
	where $u^m$ is the splitting solution of the scheme \eqref{eq:splitting} at $t=t_m$ with $u^0=u_0$.
	We now apply the Lady Windermere's fan to represent the global errors as follows
	\begin{align*}
		&u^m- u(\cdot, t_m)=\calS^m u_0-\calT^m u_0=\sum^{m-1}_{j=0} \calS^{m-j-1}(\calS-\calT)\calT^{j}u_0,\\
		&u^m_N- u^m=\calS_N^m u_0-\calS^m u_0=(I_N-\calI)u^m+\sum^{m}_{j=1} \calS_N^{m-j}(\calS_N(u^{j-1})-I_N\calS(u^{j-1})).
	\end{align*}
	Applying Theorem \ref{thm:timeerror}, we have
	\begin{align*}
		\Vert u^m- u(\cdot, t_m)\Vert_{L_{QP}^2(\bbR^d)}
		&=\Big \Vert \sum^{m-1}_{j=0} \calS^{m-j-1}(\calS-\calT)\calT^{j}u_0\Big \Vert_{L_{QP}^2(\bbR^d)}\\
		&\leq \sum^{m-1}_{j=0}\Vert \calS-\calT\Vert_{L_{QP}^2(\bbR^d)} \cdot \Vert \calT^{j}u_0 \Vert_{L_{QP}^2(\bbR^d)}\\
		&\leq C\tau^2 \sup_{0\leq t\leq T}\Vert u(\cdot,t) \Vert_{L_{QP}^2(\bbR^d)}.
		%	&= C(T) \tau^2 \max_{0\leq j\leq m-1}\Vert u(\cdot,t_j) \Vert,
	\end{align*}
	%	where $T=m\tau$ and $C(T)$ is a constant that depends on $T$.
	Considering the regularity of the solution, then
	\begin{align*}
		\Vert u^m_N- u^m \Vert_{L_{QP}^2(\bbR^d)}
		&\leq\Vert (I_N-\calI)u^m\Vert_{L_{QP}^2(\bbR^d)}
		+ \Big \Vert \sum^{m}_{j=1} \calS_N^{m-j}(\calS_N(u^{j-1})-I_N\calS(u^{j-1})) \Big \Vert_{L_{QP}^2(\bbR^d)} \\
		&\leq CN^{-\alpha} \vert u_p^m\vert_{\alpha} +
		\sum^{m}_{j=1} e^{CC_V(m-j)\tau} \Vert \calS_N(u^{j-1})-I_N\calS(u^{j-1})\Vert_{L_{QP}^2(\bbR^d)}\\
		&\leq CN^{-\alpha} \vert u_p^m\vert_{\alpha}
		+ C m\tau N^{-\alpha}\max_{0\leq j\leq m-1} \vert u_p^{j}\vert_{\alpha}\\
		& \leq CN^{-\alpha} \max_{0\leq j\leq m} \vert u_p^{j}\vert_{\alpha}.
	\end{align*}
	Therefore, the desired result can be directly obtained from the above two estimates. 

%The truncation error analysis of QSM is given in \cite{Jiang2022Numerical}. 

The above analysis framework of the PM-OS2 is also applicable to the QSM-OS2, also showing exponential convergence in space and second-order accuracy in time.
Besides, since QSM is essentially a generalization of Fourier spectral method, we can offer another way to analyze QSM. Concretely, similar to the analytical framework of Fourier spectral method to solve the Schr\"{o}dinger equation with periodic potentials \cite{Jahnke2000error}, we give the error analysis of QSM-OS2 based on the embedding theorem of quasiperiodic function space, see Appendix \ref{Appendix:analysisQSM-OS2} for details.

% \begin{remark}
	% For TQSE \eqref{eq:QSE} with an non-zero external field function $f$, the convergence analysis of the homogeneous TQSE \eqref{eq:QSE-re} in this subsection can be easily extended, and similar convergence results can be obtained.
	% \end{remark}

%\begin{remark}
%For the quasiperiodic Schr\"{o}dinger equation with an external field function $f$, %{\color{red}[?]}, 
%\textit{i.e.,}
%\begin{align*}
%i\frac{\partial u}{\partial t}=-\Delta u+Vu+f,
%\end{align*}
%the convergence analysis in this subsection is easy to be generalized and similar convergence conclusions are obtained.
%\end{remark}

\section{Numerical Results}
\label{sec:Num}

In this section, we provide some numerical experiments to verify the accuracy and performance of QSM-OS2 and PM-OS2 for solving TQSE \eqref{eq:QSE}. In the following numerical results, we only show the quasiperiodic structure in finite domain.
It also should be emphasized that the developed methods have obtained global quasiperiodic solution over $\bbR^d$.
All algorithms are implemented using MSVC++ 14.29 on Visual Studio Community 2019. The software of FFTW 3.3.5 is used to compute FFT \cite{frigo2005design}. All computations are performed on a workstation with an Intel Core 2.30GHz CPU, 16GB RAM. Let CPU(s) denote the  computational time in seconds. For QSM-OS2, let $k=5$ in \eqref{eq:taylorexp}. 

%In the following example, we verify the convergence rate in space and time directions. 
We use $L_{QP}^2$-norm to measure the numerical error 
\begin{align*}
	\Err^\tau_{h}=\Vert u^m_N- u(\cdot,t_m)\Vert_{L_{QP}^2(\bbR^d)},
\end{align*}
where $\tau$ is the time step size and $h$ is the mesh size of the $n$-dimensional torus.
The error order in time direction is calculated by 
\begin{align*}
	\Ord=\frac{\ln(\Err^{\tau_1}_{h}/\Err^{\tau_2}_{h})}{\ln(\tau_{1}/\tau_{2})}.
\end{align*}
%where $\Err^{\tau_i}_{h}$ ($i=1,2$) is the numerical error with step $\tau_i$.
Our numerical examples focus on verifying the accuracy in space direction and the error order in time direction. As a result, the final time $T$ can be arbitrarily chosen. For simplicity, we choose $T=0.001$.

\subsection{One-dimensional case}

Consider one dimensional TQSE \eqref{eq:QSE} with the incommensurate potential
\begin{align*}
	V(x)=2(\cos x+\cos\sqrt{3}x),	
\end{align*}
and initial value 
\begin{align*}
	u_0(x)=\sum_{\lambda\in \sigma(u_0)}\hat{u}_{\lambda}e^{i\lambda x},
\end{align*}
where $\hat{u}_{\lambda}=e^{-(\vert m_1\vert +\vert m_2\vert)}$ and $\sigma(u_0)=\{\lambda = m_1+m_2\sqrt{3}: m_1,m_2\in \bbZ, -32\leq m_1,m_2\leq 31\}$.

The reference solution on $\bbR$ is obtained by using PM-OS2 with the time step size $\tau=1\times 10^{-8}$ and the mesh size of the two-dimensional torus $h=\pi/128$ in each direction.
Figure \ref{fig:PMreal1d} only shows the distribution of the probability density function $\vert u(x,T) \vert^2$ in $x\in[0,240\pi]$ and corresponding diffraction pattern of the reference solution, respectively. 
In the right plot of Figure \ref{fig:PMreal1d}, we can see that the dots of different colors are constantly alternating and the distribution Fourier exponent has rotational symmetry but no translational symmetry.

%Numerical results are presented in Table \ref{tab:error} and Table \ref{tab:timeerror}, demonstrating the exponential convergence rates in space, and the second-order accuracy of OS2 as proved in Section \ref{sec:converanaly}, respectively.

\begin{figure}[!htbp]
\centering
\subfigure{
\includegraphics[width=2in]{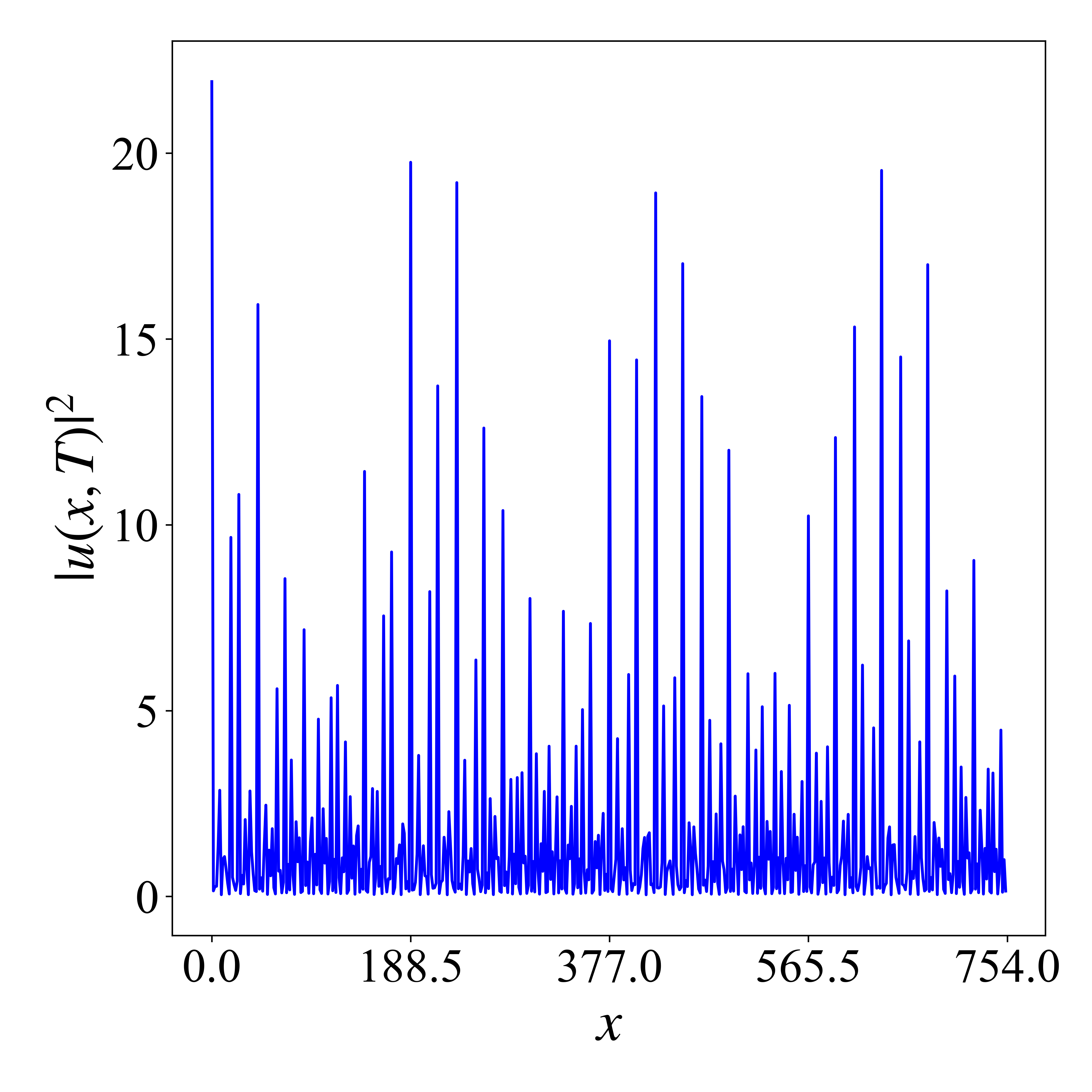}}
\hspace{0.01in}
\subfigure{
\includegraphics[width=2in]{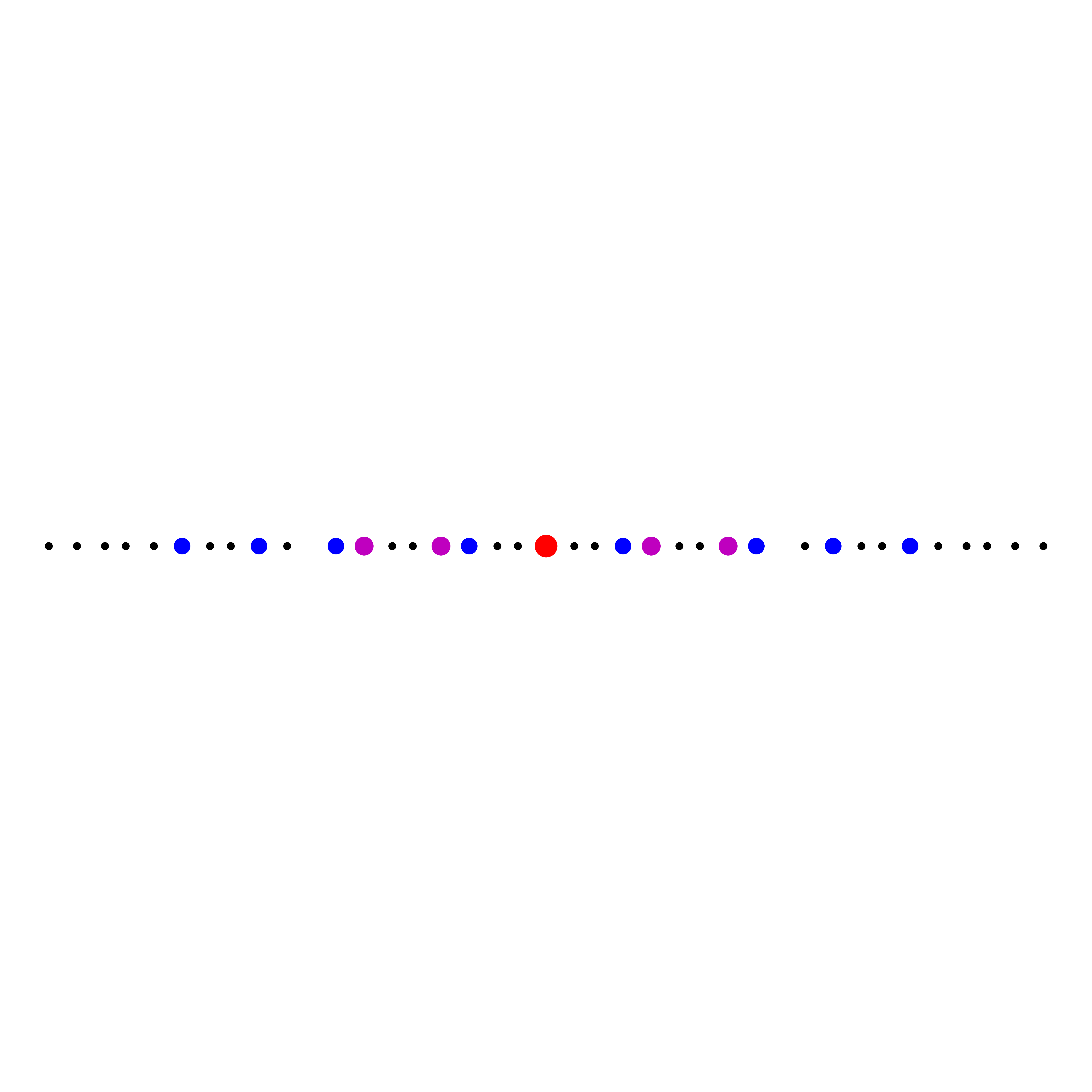}
\label{subfig:1dreciprocal}}
\caption{
The reference solution of one-dimensional TQSE. Left: Probability density distribution $\vert u(x,T) \vert^2$, $x\in[0,240\pi]$. Right: Distribution of Fourier exponents $\bf\lambda$ when $\vert \hat{u}_{\bf\lambda}\vert >8.620\mbox{e-}02$. We arrange the values of $\vert \hat{u}_{\bf\lambda}\vert$ from largest to smallest, then red, purple and blue dots represent the order of $\vert \hat{u}_{\bf\lambda}\vert$ at the corresponding Fourier exponents from the largest to the smallest, respectively.} \label{fig:PMreal1d}
\end{figure}

We first show the convergence rate when solving TQSE using PM and QSM in space direction. Set the time step size $\tau=1\times 10^{-6}$,
Table \ref{tab:error} shows numerical error $\Err^\tau_{h}$ and the required CPU time of PM-OS2 and QSM-OS2, respectively. We find that $\Err^\tau_{h}$ decays exponentially as $N$ increases due to the smoothness of wave function $u(x,t)$. Meanwhile, the CPU time of PM is less than that of QSM since PM can use 2D FFT. Figure \ref{fig:CPU_compare} further details CPU time of both methods as $N$ increases. 

%Meanwhile, the aliasing error in PM is always too small to be almost negligible.

%This is consistent with the convergence analysis given in Section \ref{sec:converanaly}.
%\begin{table}[!hptb]
%	\vspace{-0.2cm}
%	\centering
%	\footnotesize{
	%		\caption{Numerical error $\Err^\tau_{h}$ and CPU time of PM and QSM, and aliasing error $\Vert R_N\psi\Vert$ of PM for different $N$.}
	%		\vspace{0.1cm}
	%		\label{tab:error}
	%		\vspace{0.002cm}
	%		\renewcommand\arraystretch{1.3}
	%		\setlength{\tabcolsep}{3mm}
	%		{\begin{tabular}{|c|c|c|c|c|c|}\hline
			%		&		$2N \times 2N$& $4\times 4$   & $8\times 8$ &$16\times 16$  &$32\times 32$    \\ \hline
			%	\multirow{3}*{$\Err^\tau_{h}$ } &PM   &  5.091e-04    &  1.696e-05   &   1.137e-08  &   2.086e-11  \\ \cline{2-6}
			%		&		QSM            & 5.091e-04       &  1.696e-05          & 1.137e-08         &2.092e-11       \\ \cline{2-6}
			%	    &   $\Vert R_N\psi\Vert$& 1.415e-13       &  1.193e-13          & 9.609e-14         & 1.183e-13            \\ \hline
			%		\end{tabular}}
	%	}
%\end{table}

\begin{table}[!hptb]
	\vspace{-0.2cm}
	\centering
\caption{Numerical error $\Err^\tau_{h}$ and CPU time of PM-OS2 and QSM-OS2 for different $N$ and $\tau=1\times 10^{-6}$.}
\vspace{0.2cm}
\label{tab:error}
\renewcommand\arraystretch{1.6}
\setlength{\tabcolsep}{2.8mm}
{\begin{tabular}{|c|c|c|c|c|c|c|}\hline
&$N \times N$ &$2\times 2$  & $4\times 4$   & $8\times 8$ &$16\times 16$  &$32\times 32$    \\ \hline
\multirow{2}*{$\Err^\tau_{h}$ } &PM-OS2    &  2.784e-03  &  5.091e-04    &  1.696e-05   &   1.137e-08  &  2.488e-12  \\ \cline{2-7}
&QSM-OS2    & 3.335e-03     & 5.430e-04       &  1.748e-05 & 1.153e-08   & 2.485e-12     \\ \hline
%\multirow{2}*{CPU(s)} & PM-OS2 &0.042      &    0.062  &  0.112     & 0.426    &1.630   \\ \cline{2-7}
 %                     &  QSM-OS2&0.074  &  0.300   &  2.878     &55.264    &2002.697  \\ \hline
\end{tabular}}
\end{table}

\begin{figure}[!htbp]
\centering
\includegraphics[width=2in]{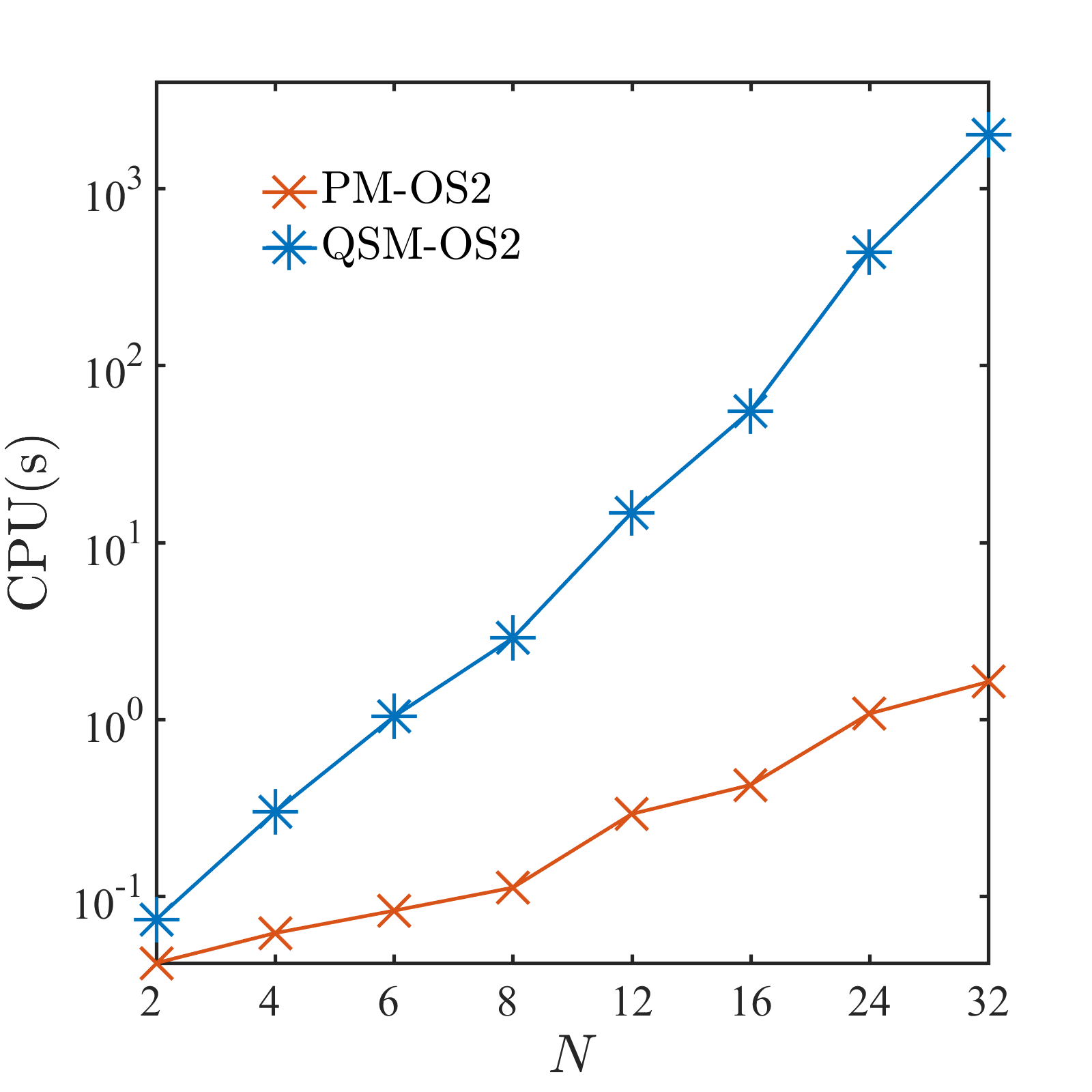}
\caption{CPU time between PM-OS2 and QSM-OS2 against $N$.}\label{fig:CPU_compare}
\end{figure}

Next, we verify the error order of applying OS2 method to solve TQSE in time direction. Meanwhile, we choose $h=\pi/128$ to make the approximation error in space direction that does not affect the error in time direction. From the numerical results shown in Table \ref{tab:error}, PM-OS2 and QSM-OS2 have the same convergence rate in space direction, while the PM-OS2 is more efficient in CPU time. Therefore, the time error of both methods is the same and we only present the temporal error of PM-OS2, see Table \ref{tab:timeerror}.

\begin{table}[!hptb] 
	\vspace{-0.2cm}
	\centering
\caption{Temporal error test of PM-OS2 with $h=\pi/128$. }\label{tab:timeerror}
\vspace{0.1cm}
\renewcommand\arraystretch{1.6}
\setlength{\tabcolsep}{3mm}
{\begin{tabular}{|c|c|c|c|c|}\hline
$\tau$   &$1\times 10^{-3}$   &$5\times 10^{-4}$&$2.5\times 10^{-4}$  & $1.25\times 10^{-4}$\\ \hline
$\Err^\tau_{h}$ & 1.608e-09 & 4.021e-10 & 1.005e-10 & 2.513e-11  \\ \hline
$\Ord$& - & 2.00 &2.00   & 2.00    \\ \hline
		\end{tabular}}
\end{table}
%\begin{table}[!hptb] 
%	\vspace{-0.2cm}
%	\centering
%	\footnotesize{
	%		\caption{Required CPU time (s) of PM and QSM for different $N$.}\label{tab:CPU}
	%		\vspace{0.1cm}
	%		\vspace{0.001cm}
	%		\renewcommand\arraystretch{1.3}
	%		\setlength{\tabcolsep}{3mm}
	%		{\begin{tabular}{|c|c|c|c|c|c|c|c|}\hline
			%				$2N\times 2N$  & $4\times 4$& $8\times 8$ &$16\times 16$  &$32\times 32$ &$64\times 64$   \\ \hline
			%				PM   &0.006 & 0.021  &  0.097 & 0.285    &         \\ \hline
			%				QSM  &0.106 & 1.440   & 22.101  &351.243    &  \\ \hline
			%		\end{tabular}}
	%	}
%\end{table}

%\begin{example}
%Consider two dimensional QSE \eqref{eq:QSE} with the quasiperiodic potential
%\begin{align*}
%V(\bx)=\sum_{\blam\in \bLam_1}e^{i\blam\cdot \bx},~~\blam_1=\{(1,0),(\sqrt{5},1)\}.	
%\end{align*}
%\end{example}

\subsection{Two-dimensional cases}

%\begin{align*}
%	V(x)=2(\cos x_1+\cos\sqrt{5}x_1+\cos\sqrt{2}x_2).
%\end{align*}

In this subsection, we further demonstrate that the PM is a high-precision and efficient algorithm to solve TQSE through the two-dimensional example.
Consider the potential function
\begin{align*}
	V(\bx)=\sum_{j=1}^{5} e^{i (\bP_1\bk_j)\cdot \bx}-e^{i (\bP_1\bk_6)\cdot \bx},	~~\bx\in\bbR^2,
\end{align*}
where 
\begin{align*}
	(\bk_1, \bk_2, \bk_3,\bk_4,\bk_5,\bk_6)= 
	\begin{pmatrix}
		0 & 0 & 1 & 0 & 0 & 0\\ 
		1 & -1& 0 & 0 & 0 & 0\\
		0 & 0 & 0 & 0 & 1 & -1\\
		0 & 0 & 0 & 1 & 0 & 0
	\end{pmatrix},
\end{align*}
and the corresponding projection matrix
\begin{align*}
	\bP_1 = \begin{pmatrix}
		1 & \cos(\pi/4) & 0 & -\cos(\pi/4) \\
		0 & \sin(\pi/4) & 1 & \sin(\pi/4) 
	\end{pmatrix}.
\end{align*} 
Therefore, this quasiperiodic system can be embedded into a four-dimensional parent system. 
Let the initial value
\begin{align*}
	u_0(\bx)=\sum_{\blam\in \sigma_1(u_0)}e^{-(\vert k_1\vert +\vert k_2\vert+\vert k_3\vert+\vert k_4\vert )} e^{i\blam\cdot \bx},~~\blam = \bP_1\bk,
\end{align*}
where $\sigma_1(u_0)=\{\blam = \bP_1\bk: k_1, k_2, k_3, k_4\in \bbZ, -16\leq k_1, k_2, k_3,k_4\leq 15\}.$

The reference solution on $\bbR^2$ is given by PM-OS2 under $\tau=1\times 10^{-7}$ and the mesh size of the four-dimensional torus $h=\pi/64$ in each direction. In Figure \ref{fig:PMreal2d8}, we present the distribution of the probability density function $\vert u(\bx,T) \vert^2$ and the corresponding diffraction pattern of the reference solution, respectively. 
The structure shown in Figure \ref{fig:PMreal2d8} is a typical octagonal quasicrystal. The reciprocal space, shown on the right plot of Figure \ref{fig:PMreal2d8}, indicates  octagonal rotational symmetry.

\begin{figure}[!htbp]
	\centering
	\subfigure{
		\includegraphics[width=2in]{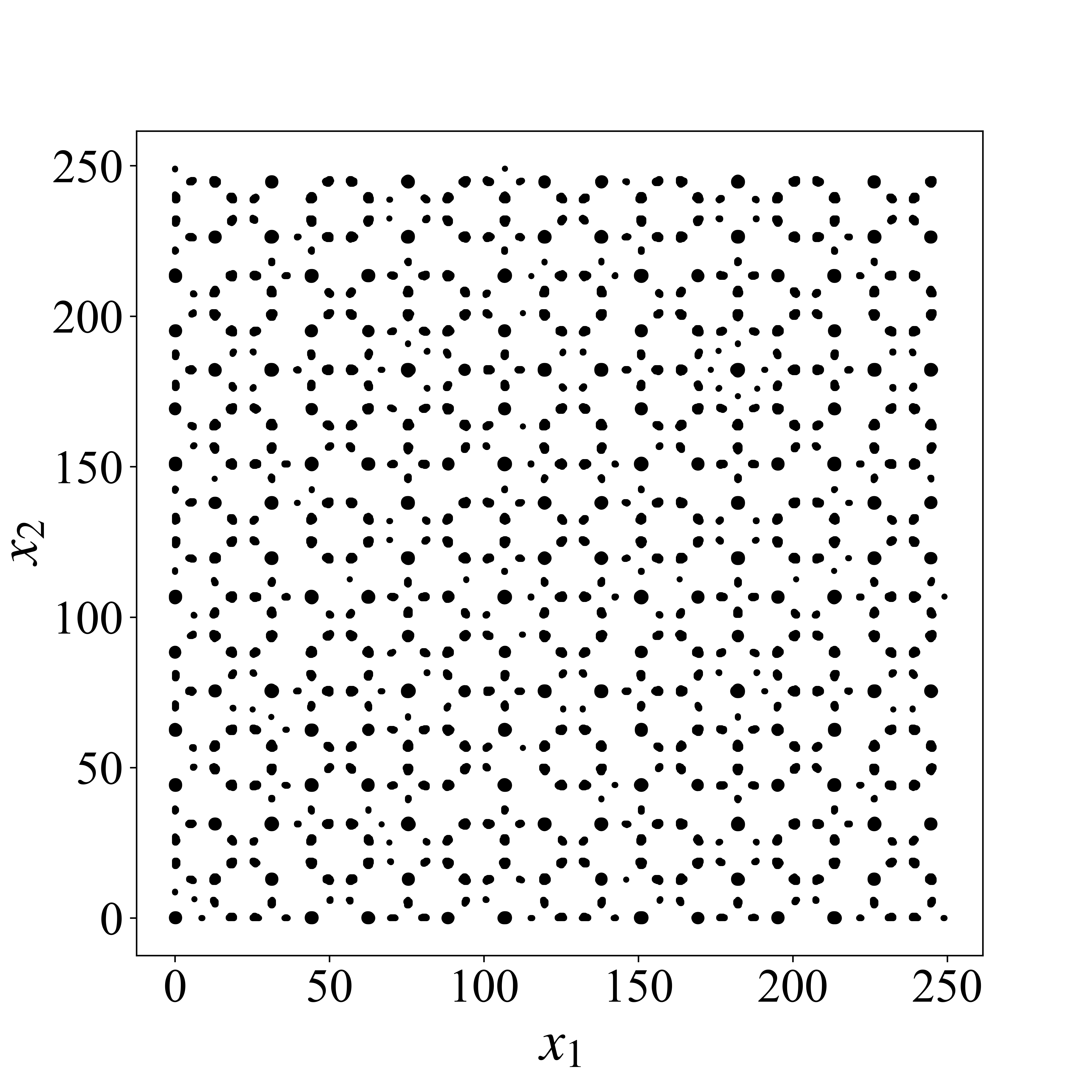}}
	\hspace{0.01in}
	\subfigure{
		\includegraphics[width=2in]{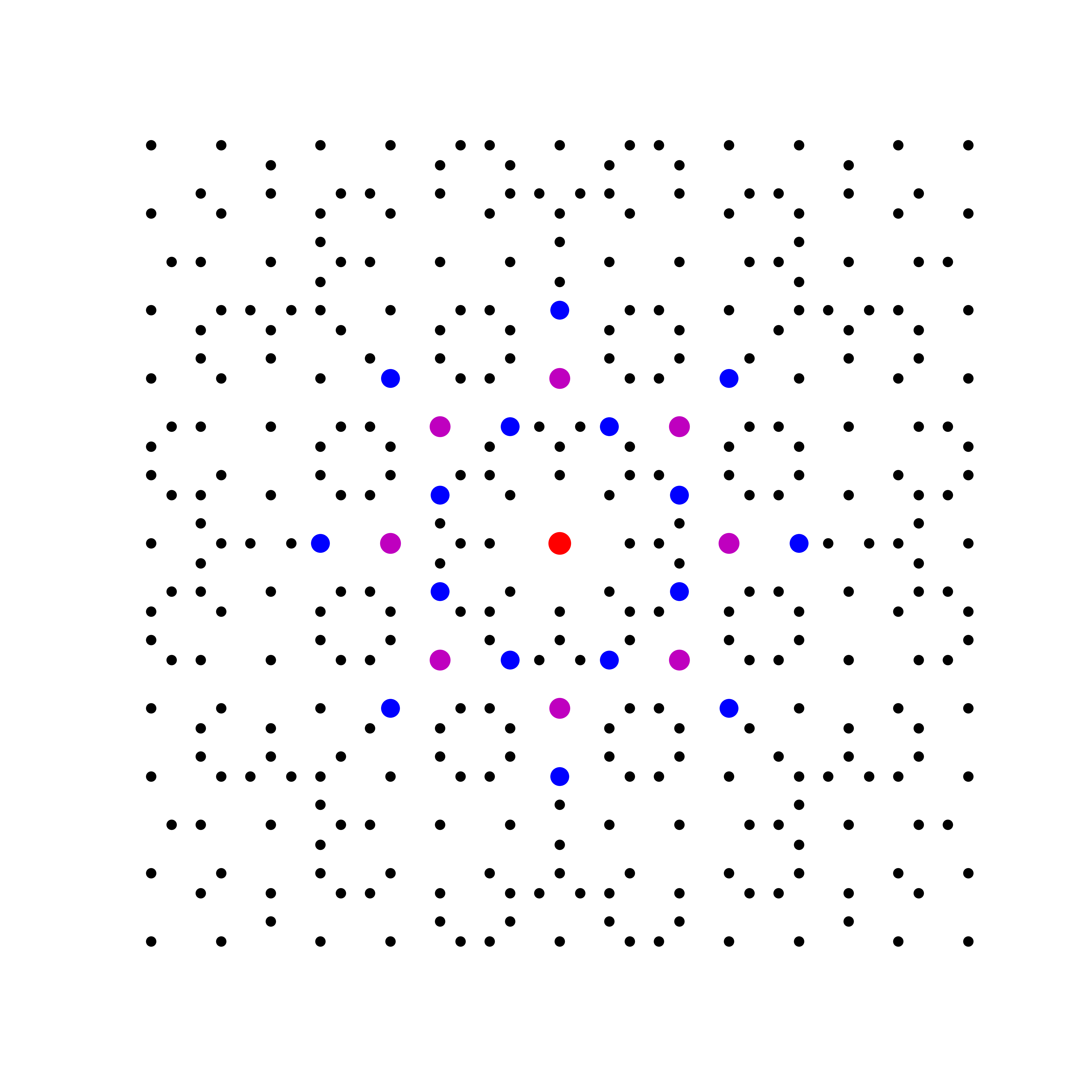}
		\label{subfig:2D8fold}}
	\caption{The reference solution of two-dimensional TQSE. 
		Left: Probability density distribution $\vert u(\bx,T) \vert^2$, $\bx\in[0,80\pi)^2$. Right: Distribution of Fourier exponents $\blam$ when $\vert \hat{u}_{\blam} \vert >23.970$. By arranging the values of $\vert \hat{u}_{\blam}\vert$ from largest to smallest, red, purple and blue dots represent the distribution of Fourier exponents corresponding to the first, first nine and first twenty-five values of $\vert \hat{u}_{\blam}\vert$, respectively.
	} \label{fig:PMreal2d8}
\end{figure}

Table \ref{tab:errorPM} shows the numerical error $\Err^\tau_{h}$ and CPU time of PM-OS2. Obviously, the PM-OS2 has exponential convergence. When the fine mesh size $h=\pi/16$, the numerical error in space does not affect the error in time direction. Table \ref{tab:timeerror2d} verifies that the PM-OS2 has second-order error accuracy in time direction with $h=\pi/16$. 

\begin{table}[!hptb]
	\vspace{-0.2cm}
	\centering
		\caption{Numerical error $\Err^\tau_{h}$ and CPU time of PM-OS2 for different $N$.}
		\vspace{0.1cm}
		\label{tab:errorPM}
		\renewcommand\arraystretch{1.6}
		\setlength{\tabcolsep}{2.8mm}
		{\begin{tabular}{|c|c|c|c|c|}\hline
				$N \times N\times N\times N$& $4\times 4\times 4\times 4$   & $8\times 8\times 8\times 8$ &$16\times 16\times 16\times 16$  &$32\times 32\times 32\times 32$\\ \hline
				$\Err^\tau_{h}$ &3.881e-03   &4.830e-04   &9.936e-09 & 7.174e-13 \\ \hline
				CPU(s) &0.132    &1.537  & 32.217   &663.971  \\ \hline
		\end{tabular}}
\end{table}

\begin{table}[!hptb] 
\vspace{-0.2cm}
\centering
\caption{Temporal error test of PM-OS2 with $h=\pi/16$. }\label{tab:timeerror2d}
		\vspace{0.2cm}
		\renewcommand\arraystretch{1.6}
		\setlength{\tabcolsep}{3mm}
		{\begin{tabular}{|c|c|c|c|c|}\hline
$\tau$   &$1\times 10^{-3}$   &$5\times 10^{-4}$&$2.5\times 10^{-4}$  & $1.25\times 10^{-4}$\\ \hline
$\Err^\tau_{h}$ &5.230e-10  &1.308e-10  &3.281e-11   &8.667e-12      \\ \hline
$\Ord$ &- & 2.00     &1.99  &1.92  \\ \hline
		\end{tabular}}
\end{table}

Furthermore, by changing the projection matrix and the initial value, we can obtain different quasicrystal.
Let projection matrix be
\begin{align*}
	\bP_2 = \begin{pmatrix}
		1 & \cos(\pi/6) & \cos(\pi/3) & 0\\
		0 & \sin(\pi/6) & \sin(\pi/3) & 1
	\end{pmatrix},
\end{align*}
and initial value be 
\begin{align*}
	u_0(\bx)=\sum_{\blam\in \sigma_2(u_0)}e^{-\Vert \blam\Vert^2} e^{i\blam\cdot \bx},
\end{align*}
where $\sigma_2(u_0)=\{\blam = \bP_2\bk: k_1, k_2, k_3, k_4\in \bbZ, -8\leq k_1, k_2, k_3, k_4\leq 7\}$, we can obtian the dodecagonal quasicrystal as shown in Figure \ref{fig:PMreal2d12} under the time step size $\tau=1\times 10^{-7}$ and the mesh size $h=\pi/32$ of the four-dimensional torus.
%{\color{blue}In Figure \ref{fig:PMreal2d12}, we present the distribution of the probability density function $\vert u(\bx,T) \vert^2$ of the reference solution in $[0,120\pi)^2$.}

%{\color{red}Therefore, PM is highly accurate and efficient in solving high-dimensional TQSE.}

\begin{figure}[!htbp]
	\centering
	\subfigure{
\includegraphics[width=2in]{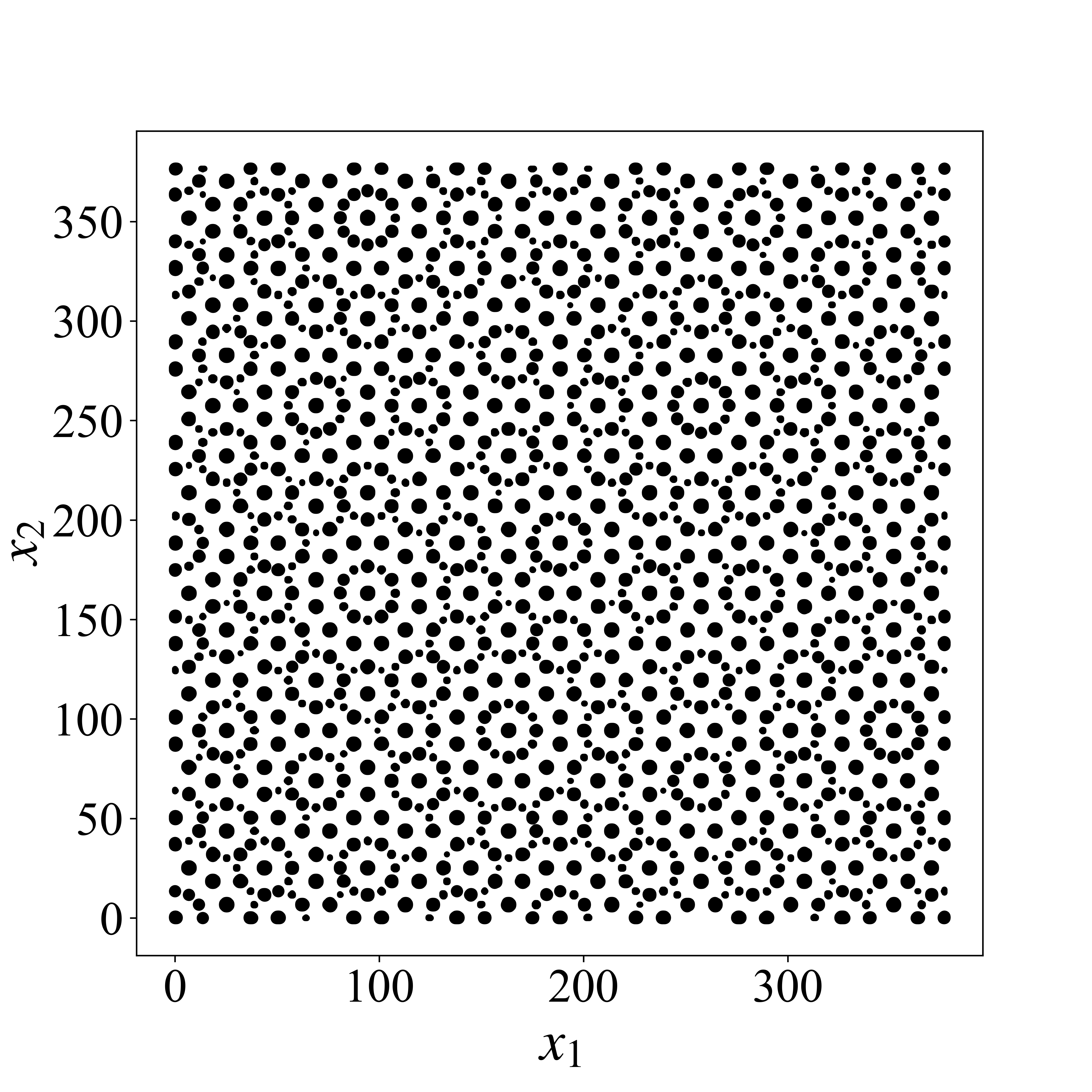}}
	\hspace{0.01in}
	\subfigure{
\includegraphics[width=2in]{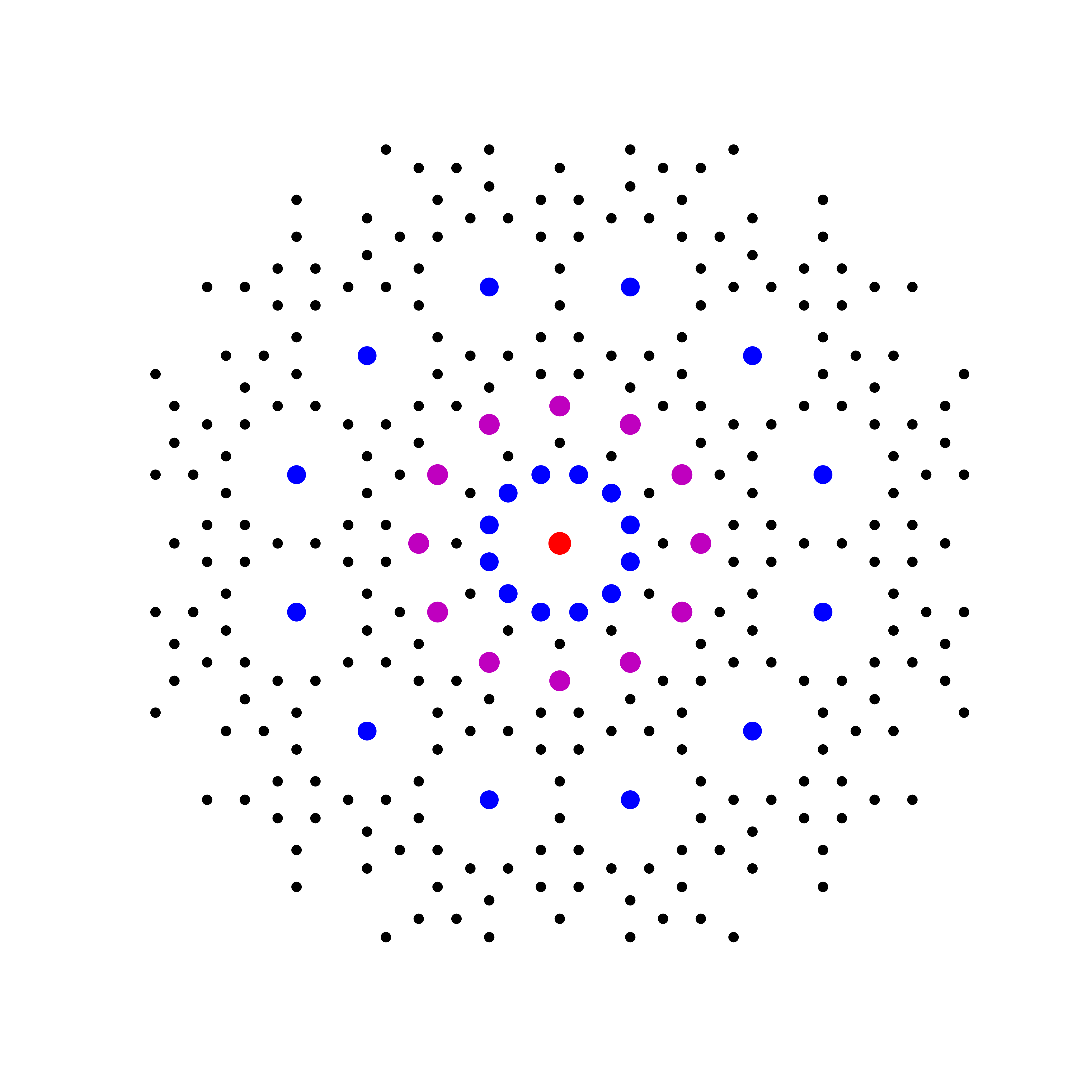}}
\caption{The reference solution of two-dimensional TQSE. 
Left: Probability density distribution $\vert u(\bx,T) \vert^2$, $\bx\in[0,120\pi)^2$. Right: Distribution of Fourier exponents $\blam$ when $\vert \hat{u}_{\blam} \vert >0.044$. We arrange the values of $\vert \hat{u}_{\bm\lambda}\vert$ from largest to smallest, then red, purple and blue dots represent the distribution of Fourier exponents corresponding to the first, first thirteen and first thirty-seven values of $\vert \hat{u}_{\bm\lambda}\vert$, respectively.}\label{fig:PMreal2d12}
\end{figure}

\section{Conclusions}
\label{sec:diss}

In this paper, two high-accuracy numerical methods, the QSM-OS2 and PM-OS2, have been developed to solve arbitrary dimensional TQSE \eqref{eq:QSE} and to obtain global quasiperiodic solutions. 
A rigorous convergence analysis 
shows that QSM-OS2 and PM-OS2 have exponential convergence in space when the potential function has enough regularity and second-order accuracy in time. Meanwhile the computational complexity analysis demonstrates that the PM-OS2 is more efficient than the QSM-OS2. The one- and two-dimensional numerical experiments further verify the  theoretical results.

%\section*{Declarations}
%
%%{\bf Funding:} 
%This work was supported in part by the National K\&D Program of China (2023YFA1008802, 2023YFB3001604), National Natural Science Foundation of China (12171412).
%JZ was supported in part by Hunan Youth Science and Technology Innovation Talents Project (2021RC3110), the Key Project of Education Department of Hunan Province (21A0116).

%\bmhead{Acknowledgments}
%
%Acknowledgments are not compulsory. Where included they should be brief. Grant or contribution numbers may be acknowledged.
%
%Please refer to Journal-level guidance for any specific requirements.

%\begin{thebibliography}{99} 

%\end{thebibliography}

%\begin{itemize}
%\item Conflict of interest/Competing interests (check journal-specific guidelines for which heading to use)
%\item Ethics approval 
%\item Consent for publication
%\item Availability of data and materials
%\item Authors' contributions
%\end{itemize}

%\noindent
%If any of the sections are not relevant to your manuscript, please include the heading and write `Not applicable' for that section. 

\begin{appendices}

\section{Another way to analyze QSM-OS2}
\label{Appendix:analysisQSM-OS2}

The QSM is an extension of Fourier spectral method. 
%From numerical implementation of QSM-OS2, we can find that QSM-OS2 is implemented in the space of quasiperiodic functions, independent of higher dimensional parent function. 
Based on the embedding theorem of the quasiperiodic function given below, we can give the convergence analysis of QSM-OS2.
Firstly, we give the following embedding theorem.
Similar to the definition of $X_{\alpha}$ on a torus $\bbT^n$, we define the quasiperiodic function space $\mathcal{X}_{\alpha}$ on $\bbR^d$
\begin{align*}
	\calX_{\alpha}=
	\Big \{f(\bx)=\sum_{\blam\in\sigma(f)} \hf_{\blam}e^{i\blam\cdot \bx}\in L_{QP}^2(\bbR^d):
	\Vert  (-\Delta)^\alpha f \Vert^2
	=\sum_{\blam\in \sigma(f)} \vert \hf_{\blam}\vert^2\cdot
	\Vert \blam \Vert^{4\alpha} <\infty \Big \}.
\end{align*}

\begin{thm}
	\label{thm:embedingquasiperiodic}
	For any $f\in \QP(\bbR^d)$. Assume that $\alpha\geq s/2> d/4$. Then, the bound
	\begin{align*}
		\Vert f\Vert_{L_{QP}^{\infty}(\bbR^d)}\leq C\, \Vert f\Vert_{s} \leq C\, \Vert f\Vert_{\calX_{\alpha}}
	\end{align*} 
	is valid and the following estimate
	\begin{align*}
		\Vert w f \Vert_{L_{QP}^2(\bbR^d)}\leq C \, \Vert w \Vert_{L_{QP}^2(\bbR^d)}\cdot\Vert f\Vert_{\calX_\alpha},
		~~w\in L_{QP}^{2}(\bbR^d),~~f\in \calX_{\alpha}
	\end{align*}
	holds where $C$ is a constant.
\end{thm}

\begin{proof}
	By the H\"{o}lder inequality, we have 
	\begin{align*}
		\sum^{p}_{j=1}\vert \hat{f}_{\blam_j}\vert
		&=\sum^{p}_{j=1}\vert \hat{f}_{\blam_j}\vert \cdot (1+\vert \blam_j \vert^2)^{s/2} \cdot (1+\vert \blam_j \vert^2)^{-s/2}\\
		&\leq \Big( \sum^{p}_{j=1}\vert \hat{f}_{\blam_j}\vert^{2}
		(1+\vert \blam_j \vert^2)^{s} \Big )^{\frac{1}{2}}
		\cdot \Big(\sum^{p}_{j=1} (1+\vert \blam_j \vert^2)^{-s}\Big )^{\frac{1}{2}}.
	\end{align*}
	Set $p\rightarrow \infty$, it follows that
	\begin{align*}
		\sum^{\infty}_{j=1}\vert \hat{f}_{\blam_j}\vert
		&\leq \Big( \sum^{\infty}_{j=1}\vert \hat{f}_{\blam_j}\vert^{2}
		(1+\vert \blam_j \vert^2)^{s} \Big )^{\frac{1}{2}}
		\cdot \Big( \sum^{\infty}_{j=1} (1+\vert \blam_j \vert^2)^{-s}\Big )^{\frac{1}{2}}\\
		&=\Vert f \Vert_{H^s_{QP}(\bbR^d)}
		\cdot \Big( \sum^{\infty}_{j=1} (1+\vert \blam_j \vert^2)^{-s}\Big )^{\frac{1}{2}}.
	\end{align*}
	When $s>d/2$, the series 
	$\sum^{\infty}_{j=1} (1+\vert \blam_j \vert^2)^{-s}$ converges. 
	For $f\in H^s_{QP}(\bbR^d)$, then $f\in L^1_{QP}(\bbR^d)$ and the Bohr-Fourier series of $f$ is absolutely convergence. Then it converges uniformly to $f$. 
	Consequently, 
	\begin{align*}
		\Vert f\Vert_{L_{QP}^{\infty}(\bbR^d)}
		\leq \sum^{\infty}_{j=1}\vert \hat{f}_{\blam_j}\vert
		\leq C \Vert f \Vert_{H^s_{QP}(\bbR^d)},
	\end{align*}
	where $C$ is a constant. Similar to the proof of Theorem \ref{lem:normineq}, this theorem is proved.
\end{proof}

The error analysis of QSM without the help of parent functions is given below, see \cite{Jiang2022Numerical} for details. 

\begin{thm}[\cite{Jiang2022Numerical}]
	\label{thm:QSM-analysis2}
	Suppose that $f\in H^{\alpha}_{QP}(\bbR^d)$ and the nonzero minimum singular value $\sigma_{\min}(\bm P)$ of the projection matrix $\bm P$ satisfies $\sigma_{\min}(\bm P)>\theta >0$. Then, there exists a constant $C(\theta)$, independent of $f$ and $N$, such that		
	\begin{align*}
		\Vert\mathcal{P}_N f-f\Vert_{L_{QP}^2(\bbR^d)} \leq C(\theta) N^{-\alpha}\vert f \vert_\alpha.
	\end{align*}
\end{thm}

Jahnke \textit{et al.} have proved the convergence analysis by applying OS2 method in time and the Fourier pseudo-spectral method in space to solve the Schr\"{o}dinger equation with periodic potentials \cite{Jahnke2000error}. Therefore, based on Theorems \ref{thm:embedingquasiperiodic} and \ref{thm:QSM-analysis2}, similar to the analysis in \cite{Jahnke2000error}, we can obtain the error analysis of QSM-OS2.

\begin{thm}
	Let $u(\cdot,t_m)$ and $u^m_N$ be the solutions of problems \eqref{eq:QSE-re} and \eqref{eq:splitting-QSM} at $t_m$, respectively. Then under the conditions
	
	(i) The potential $v$ is a $\calC^1$-smooth function and $\Vert v\Vert_{\calX_\alpha}\leq C$, $\alpha\geq s/2> d/4$;
	
	(ii) The quasiperiodic function $u^j\in H_{QP}^{\alpha}(\bbR^d),~0\leq j\leq m$;\\
	the global error bound of QSM-OS2 \eqref{eq:splitting-QSM} is	
	\begin{align*}
		\Vert u^m_N- u(\cdot,t_m)\Vert_{L_{QP}^2(\bbR^d)}\leq C(\tau^2 + N^{-\alpha}).
	\end{align*}
The constant $C>0$ depends on $C_V$, $\sup\{\Vert u(\cdot, t) \Vert_{L_{QP}^2(\bbR^d)}: 0\leq t \leq T \}$ and $\max\{\vert u^j\vert_{\alpha}: 0\leq j\leq m \}$.
\end{thm}




\end{appendices}



\begin{thebibliography}{99} 

	\bibitem{Dirac1958principles}
	P. Dirac, The principles of quantum mechanics, Oxford University Press, 1958.
	
	\bibitem{Jahnke2000error}
	T. Jahnke and C. Lubich, Error bounds for exponential operator splittings, BIT Numerical Mathematics, 40(4): 735-744, 2000.	
	
	\bibitem{Jitomirskaya1999Metal-insulator}
	S. Jitomirskaya, Metal-insulator transition for the almost Mathieu operator, Ann. Math., 150: 1159-1175, 1999.
	
	
\bibitem{Nixon2010Electronic}
L. Nixon and D. Papaconstantopoulos,
Electronic structure and superconductivity of europium, Physica C: Superconductivity and its Applications, 47: 17-18, 2010.
	
	\bibitem{Zhou2014Particle-hole}
	L. Zhou and J. Carbotte, Particle-hole asymmetry on Hall conductivity of a topological insulator, Phys. Rev. B, 89(5): 085413, 2014.
	
	\bibitem{wang2020localization}	
	P. Wang, Y. Zheng, X. Chen, C. Huang, Y. Kartashov, L. Torner, V. Konotop and F. Ye, Localization and delocalization of light in photonic Moir\'{e} lattices, Nature, 577: 42-46, 2020.
	
	%\bibitem{Efremidis2002discrete}
	%N. Efremidis, S. Sears, D. Christodoulides, J. Fleischer, and M. Segev, Discrete solitons in photorefractive optically induced photonic lattices, Phys. Rev. E, 66: 046602, 2002.
	
	\bibitem{Kohmoto1983Metal-Insulator}
	M. Kohmoto, Metal-Insulator transition and scaling for incommensurate systems, Phys. Rev. Lett., 51(13): 1198-1201, 1983.
	
	
	\bibitem{Lahini2009observation}
	Y. Lahini, R. Pugatch, F. Pozzi, M. Sorel, R. Morandotti, N. Davidson and Y. Silberberg, Observation of a localization transition in quasiperiodic photonic lattices,
	Phys. Rev. Lett., 103: 013901, 2009.
	
	%\bibitem{Stampfli1986dodecagonal}
	%P. Stampfli, A dodecagonal quasiperiodic lattice in two dimensional. Helv. Phys. Acta, 59:
	%1260-1263, 1986.
	
	%
	%\bibitem{Huang2016Localization}
	%C. Huang, F. Ye, X. Chen, Y. V Kartashov, V. V Konotop and L. Torner, Localization-delocalization wavepacket transition in Pythagorean aperiodic potentials. Sci. Rep.,  6: 32546, 2016.
	
	\bibitem{Cao2018unconventional}
	Y. Cao, V. Fatemi, S. Fang, K. Watanabe, T. Taniguchi,
	E. Kaxiras and P. Jarillo-Herrero, Unconventional super-conductivity in magic-angle graphene superlattices, Nature, 43: 556, 2018.
	
	\bibitem{Rammal1983superconducting}
	R. Rammal, T. Lubensky and G. Toulouse, Superconducting diamagnetism near the percolation threshold, Journal De Physique Lettres, 44: 65-71, 1983.
	
	\bibitem{Merlinl1985Quasiperiodic}
	R. Merlin, K. Bajema, Roy Clarke, F. Juang and P. Bhattacharya, Quasiperiodic GaAs-AlAs Heterostructures, Phys. Rev. Lett., 55: 1768, 1985.
	
	
	
	% \bibitem{Kohmotol1986Quasiperiodic}
	% M. Kohmoto and J. Banavar, Quasiperiodic lattice: Electronic properties, phonon properties, and diffusion, Phys. Rev. B, 34: 563, 1986.
	
	\bibitem{Kuksin1996Invariant}
	S. Kuksin and J. P\"{o}schel, Invariant Cantor manifolds of quasiperiodic oscillations for a nonlinear Schr\"{o}dinger equation, Ann. Math., 143: 149-179, 1996.
	
	
	\bibitem{Kuksin1987Hamiltonian}
	S. Kuksin, Hamiltonian perturbations of infinite-dimensional linear systems with an imaginary spectrum, Funct. Anal. Appl., 21: 192-205, 1987.
	
	
	
	
	% \bibitem{Liang2005Quasi-periodic}
	% Z. Liang and J. You, Quasi-periodic solutions for 1D Schr\"{o}dinger equation with higher order nonlinearity, Siam J. Math. Anal. 36: 1965-1990, 2005.
	
	% \bibitem{Bourgain1996Gibbs}
	%J. Bourgain, Gibbs measures and quasi-periodic solutions for nonlinear Hamiltonian partial different equations, 23-43, Gelfand Math. Sem., Birkh\"{a}user Boston, Boston, MA, 1996.
	
	\bibitem{Bourgain1994Construction}
	J. Bourgain, Construction of quasi-periodic solutions for Hamiltonian perturbations of linear equations and applications to nonlinear PDE, Internat. Math. Res. Notices, 475-497, 1994.
	
	\bibitem{Bourgain1998Quasi-periodic}
	J. Bourgain, Quasi-periodic solutions of Hamiltonian perturbations of 2D linear Schr\"{o}dinger equations, Ann. Math., 148: 363-439, 1998.
	
	\bibitem{Cong2023Longtime}
	H. Cong, Y. Shi and W. Wang, Long-time Anderson localization for the nonlinear quasi-periodic Schr\"{o}dinger equation on $\bbZ^d$, arXiv:2309.15706.
	
	\bibitem{Berti2012Sobolev}
	M. Berti and P. Bolle, Sobolev quasi-periodic solutions of multidimensional wave equations with a multiplicative potential, Nonlinearity, 25: 257, 2012.
	
	\bibitem{Berti2013Quasi-periodic}
	M. Berti and P. Bolle, Quasi-periodic solutions with Sobolev regularity of NLS on $\bbT^d$ with a multiplicative potential, J. Eur. Math. Soc., 15: 22, 2013.
	
	
	\bibitem{Wang2020Space}
	W. Wang, Space quasi-periodic standing waves for nonlinear Schr\"{o}dinger equations, Commun. Math. Phys., 378(2): 783-806, 2020.
	
	\bibitem{Wang2022Infinite}
	W. Wang, Infinite energy quasi-periodic solutions to nonlinear Schr\"{o}dinger equations on $\bbR$, Int. Math. Res. Notices, rnab327, 2022.
	%\bibitem{Avila2015Global}
	%A. Avila, Global theory of one-frequency Schr\"{o}dinger operators. Acta Math. 215, 1-54, 2015.
	
	
	
	\bibitem{Avila2009spectrum}
	A. Avila, On the spectrum and Lyapunov exponent of limit periodic Schr\"{o}dinger operators,
	Commun. Math. Phys., 288: 907-918, 2009.
	
	%\bibitem{Avila2009ten}
	%A. Avila and S. Jitomirskaya, The ten martini problem. Ann. Math., 170: 303-342, 2009.
	
	%\bibitem{Klein2005Anderson}
	%S. Klein. Anderson localization for the discrete one-dimensional quasi-periodic Schr\"{o}dinger operator with potential defined by a Gevrey class function. J. Funct. Anal., 218: 255-292, 2005.
	
	% \bibitem{Morozov2014Complete}
	% S. Morozov, L. Parnovski, R. Shterenberg, Complete asymptotic expansion of the integrated density of states of multidimensional almost-Periodic pseudo-Differential operators, Ann. Henri Poincar\'{e}, 15: 263312, 2014. 
	
	%\bibitem{Sorets1991Discrete}
	%E. Sorets, T. Spencer. Discrete one-dimensional quasi-periodic Schr\"{o}dinger operators with pure point spectrum. Commun. Math. Phys., 142(3):543-566, 1991.
	
	
	
	%\bibitem{Jitomirskaya2016Dynamical}
	%S. Jitomirskaya and R. Mavi. Dynamical bounds for quasiperiodic Schr\"{o}dinger operators with rough potentials. Int. Math. Res. Notices, 2017(1): 96-120, 2016.
	
	
	\bibitem{Marx2017Dynamics}
	C. Marx and S. Jitomirskaya, Dynamics and spectral theory of quasi-periodic Schr\"{o}dinger-type operators, Ergod. Theor. Dyn. Syst., 37(8): 2353-2393, 2017.
	
	% {\color{blue}
		% \bibitem{Jitomirskaya2020Anderson}
		% S. Jitomirskaya, W. Liu and Y. Shi, Anderson localization for multi-frequency quasi-periodic operators on $\bbZ^d$, Geom. Funct. Anal. (GAFA) 30: 457-481, 2020.}
	
	
	%\bibitem{Sakaguchi20006Gap}
	% H. Sakaguchi and B. A. Malomed. Gap solitons in quasiperiodic optical lattices. Phys. Rev. E, 2006.
	% 
	% \bibitem{Kartashov2021Multifrequency}
	% Y. V. Kartashov, F. Ye, V. V. Konotop and L. Torner. Multifrequency Solitons in Commensurate-Incommensurate Photonic Moiré Lattices. Phys. Rev. Lett., 127(16): 163902, 2021.
	
	
	%  Zhao, X. (2018). Continuity of the spectrum of quasi-periodic Schrödinger operators with finitely differentiable potentials. Ergodic Theory and Dynamical Systems, 40, 564 - 576.
	
	% \bibitem{Shamis2011Some}
	% M. Shamis, Some connections between almost periodic and periodic discrete Schr\"{o}dinger operators with analytic potentials, J. Spectr. Theor., 3: 349-362, 2011.
	
	\bibitem{Jitomirskaya2012Analytic}
	S. Jitomirskaya and C. Marx, Analytic quasi-periodic Schr\"{o}dinger operators and rational frequency approximants, Geom. Funct. Anal., 22: 1407-1443, 2012.
	
	\bibitem{Damanik2014isospectral}
	D. Damanik, M. Goldstein and M. Lukic, The isospectral torus of quasi-periodic Schr\"{o}dinger operators via periodic approximations, Invent. Math., 207: 895-980, 2014.
	
	
	%\bibitem{Takahashi1993applicability}
	%K. Takahashi and K. Ikeda,  Applicability of symplectic integrator to classically unstable quantum dynamics. J. Chem. Phys., 99(11):8680-8694, 1993.
	%
	%\bibitem{Auer2001fourth-order}	
	%J. Auer, E. Krotscheck, and S. Chin, A fourth-order real-space algorithm for solving local Schrodinger equations, J. Chem. Phys., 115, 6841, 2001.
	%\bibitem{Stier1997modeling}
	%O. Stier,  D. Bimberg, Modeling of strained quantum wires using eight-band k.p theory. Physical Review B, 1997, 55(12):7726-7732.
	%
	%
	
	%
	%\bibitem{Kosloff1994propagation}	
	%R. Kosloff, Propagation Methods for Quantum Molecular Dynamics. Annual Review of Physical Chemistry, 45(45):145-178, 1994.
	%
	%\bibitem{Truong1992comparative}	
	%T. N. Truong, J. J.  Tanner , P. Bala , et al. A comparative study of time dependent quantum mechanical wave packet evolution methods. Journal of Chemical Physics, 96(3):2077-2084, 1992.
	
	
	
	
	%\bibitem{Chin2001fourth}
	%S. A. Chin,  C. R. Chen, Fourth order gradient symplectic integrator methods for solving the time-dependent Schrödinger equation. J. Chem. Phys., 114(17):7338-7341, 2001.
	%
	%\bibitem{C2006Spectral}
	%C. Canuto, M. Hussaini, A. Quarteroni, and T. Zang. Spectral methods: fundamentals in single domains. Springer-Verlag, Berlin Heidelberg, 2006.
	
	
	%{\color{red}
		%\bibitem{Harald2001fourth-order}
		%A. Harald, Forbert, et al. Fourth-order diffusion Monte Carlo algorithms for solving quantum many-body problems. Physical Review B, 2001.
		%
		%\bibitem{Ciftja2003fourth}
		% O. Ciftja, S. A.  Chin, Short-time-evolved wave functions for solving quantum many-body problems. Physical Review B, 13, 2003.
		%}
	
	%\bibitem{Bao2003numerical}
	%W. Bao, D. Jaksch and P. Markowich, Numerical solution of the Gross-Pitaevskii equation for Bose-Einstein condensation, J. Comput. Phys. 187: 318-342, 2003.
	%
	%\bibitem{Caliari2009high-order}	
	%M. Caliari, C. Neuhauser and M. Thalhammer, High-order time-splitting
	%	Hermite and Fourier spectral methods for the Gross-Pitaevskii equation,
	%	J. Comput. Phys., 228 (2009) 822-832.
	
	
	%\bibitem{Thalhammer2012Convergence}
	%M. Thalhammer, Convergence analysis of high-order time-splitting pseudospectral methods for nonlinear Schr\"{o}dinger equations. Siam J. Numer. Anal., 50, 3231-3258, 2012.
	
	
	% {\color{red}
		% \bibitem{Goldman1993quasicrystals}
		% A. Goldman and R. Kelton, Quasicrystals and crystalline approximants, Rev. Mod. Phys., 65: 213, 1993.
		
		% \bibitem{Lifshitz1997theoretical}
		% R. Lifshitz and D. Petrich, Theoretical model for faraday waves with multiple-frequency forcing, Phys. Rev. Lett., 79: 1261, 1997.
		% }
	
	\bibitem{Jiang2022approximation}
	K. Jiang, S. Li and P. Zhang, On the approximation of quasiperiodic functions with Diophantine frequencies by periodic functions, arXiv: 2304.04334.
	
	
	\bibitem{jiang2014numerical}
	K. Jiang and P. Zhang, Numerical methods for quasicrystals, J. Comput. Phys., 256: 428-440, 2014.
	
	\bibitem{jiang2018numerical}
	K. Jiang and P. Zhang, Numerical mathematics of quasicrystals. Proc. Int. Cong. of Math., 3, 3575-3594, 2018.
	
	\bibitem{Jiang2022Numerical}
	K. Jiang, S. Li and P. Zhang, Numerical methods and analysis of computing quasiperiodic systems, SIAM J. Numer. Anal, in press (also see arXiv: 2210.04384).
	
	\bibitem{Gao2023Pythagoras}
	Z. Gao, Z. Xu, Z. Yang and F. Ye, Pythagoras superposition principle for localized eigenstates of two-dimensional Moir\'{e} lattices, Phys. Rev. A, 108(8): 013513, 2023.
	
	\bibitem{jiang2015stability}
	K. Jiang, J. Tong, P. Zhang and A.-C. Shi, Stability of
	two-dimensional soft quasicrystals, Phys. Rev. E., 92: 042159, 2015.
	
	%\bibitem{zhou2019plane}
	%Y. Zhou, H. Chen and A. Zhou, Plane wave methods for quantum eigenvalue problems of incommensurate systems. J. Comput. Phys., 384: 99-113, 2019.
	
	
	\bibitem{cao2021computing}
	D. Cao, J. Shen and J. Xu, Computing interface with quasiperiodicity, J. Comput. Phys., 424: 109863, 2021.
	
	\bibitem{li2021numerical}
	X. Li and K. Jiang, 
	Numerical simulation for quasiperiodic quantum dynamical systems (in Chinese), Journal on Numerical Methods and Computer Applications, 42(1): 3-17, 2021.
	
	\bibitem{jiang2022tilt}
	K. Jiang, W. Si and J. Xu, 
	Tilt grain boundaries of hexagonal structures: a spectral viewpoint, SIAM J. Appl. Math., 82: 1267-1286, 2022.
	
	\bibitem{wang2022Effective}
	C. Wang, F. Liu and H. Huang, Effective model for fractional topological corner modes inquasicrystals, Phys. Rev. Lett., 129: 056403, 2022.
	
	
	\bibitem{Strang1968construction}
	G. Strang, On the construction and comparison of difference schemes, SIAM J. Numer. Anal., 5: 506-517, 1968.
	
	
	% {\color{red}
		% \bibitem{levitan1982almost}
		% B. Levitan and V. Zhikov,  Almost periodic functions and differential equations, Cambridge University Press, 1982.
		% }
	
	\bibitem{Corduneanu1989almost}
	C. Corduneanu, Almost periodic function, Second Edition, Chelsea, New York, 1989.
	
	
	
	%\bibitem{Amerio1971almost}
	%L. Amerio,  and G. Prouse. Almost periodic functions and functional equations. Springer, New York, 1971.
	
	\bibitem{Iannacci1998embedding}
	R. Iannacci, A. Bersani, G. Dell'Acqua and P. Santucci, Embedding theorems for Sobolev-Besicovitch spaces of almost periodic functions, Zeitschrift Fur Analysis Und Ihre Anwendungen, 17: 443-457, 1998.
	
	%{\color{red}	
		%	\bibitem{Trotter1959product}
		%	H. F. Trotter, On the product of semi-groups of operators, Proc. Amer. Math. Soc., 
		%	10: 545-551, 1959.
		%	
		%	
		%	\bibitem{bohr1947almost}
		%	H. Bohr. Almost periodic functions. Chelsea, New York, 1947.
		%}
	
	\bibitem{Stone1930linear}
	M. Stone, Linear transformations in Hilbert space: III. Operational methods and group theory, Proc. Natl. Acad. Sci. USA, 16(2): 172-175, 1930.
	
	
	\bibitem{frigo2005design}
	M. Frigo and S. Johnson, The design and implementation of FFTW3, Proc. IEEE, 93: 216-231, 2005.
\end{thebibliography}

\end{document}